\documentclass[11pt]{elsarticle}
\usepackage{amsfonts}
\usepackage{graphics}
\usepackage{caption2}
\usepackage{amssymb}
\usepackage{psfrag}
\usepackage{amsmath}
\usepackage{amsthm}
\usepackage{graphicx}
\usepackage{multirow}
\usepackage{exscale}
\usepackage{subfigure}
\usepackage{relsize}
\usepackage[normalem]{ulem}
\usepackage{color}
\usepackage{fullpage}

\allowdisplaybreaks[4]

\newcommand{\unl}{\textsl}

\newcommand{\bd}{\boldsymbol}

\def\x{{\boldsymbol x}}
\def\y{{\boldsymbol y}}
\def\z{{\boldsymbol z}}

\def\0{{\bf 0}}

\newtheorem{rem}{Remark}
\newtheorem{exmp}{Example}

\newtheorem{lem}{Lemma}
\newtheorem{thm}{Theorem}

\numberwithin{equation}{section}

\journal{Computer Methods in Applied Mechanics and Engineering}

\begin{document}
	
\begin{frontmatter}
	\title{A novel bond-based nonlocal diffusion model with matrix-valued coefficients in  non-divergence form and its collocation discretization}
\tnotetext[t1]{H. Tian's work is partially supported by Chinese Fundamental Research Funds for the Central Universities 202264006 and National Nature Science Foundation of China grants 11801533 and 11971482.   L. Ju's work is partially supported by U.S. National Science Foundation grant  DMS-2109633.}
\author[ad1]{Hao Tian}
\ead{ haot@ouc.edu.cn}
\author[ad1]{Junke Lu}
\ead{lujunke@stu.ouc.edu.cn}
\author[ad2]{Lili Ju\corref{cor1}}
\ead{ju@math.sc.edu}

\cortext[cor1]{Corresponding author.}
\address[ad1]{School of Mathematical Sciences, Ocean University of China, Qingdao, Shandong 266100, China.}
\address[ad2]{Department of Mathematics, University of South Carolina, Columbia, SC 29208, USA.}

\pagestyle{myheadings} \markboth{H. Tian, L. Ju, Q. Du and Q. Du}{A Gaussian kernel based nonlocal anisotropic diffusion model in non-divergence form and its collocation discretization}

\date{}

\begin{abstract}
 Existing nonlocal diffusion models are predominantly  classified into two categories: bond-based models, which involve  a single-fold integral and usually simulate isotropic diffusion, and  state-based models, which contain a double-fold integral and can additionally prototype  anisotropic diffusion. While bond-based models exhibit computational efficiency, they are somewhat limited in their modeling capabilities. In this paper, we  develop  a novel bond-based nonlocal diffusion model with matrix-valued coefficients in non-divergence form. Our approach incorporates the coefficients into a covariance matrix and employs the multivariate Gaussian function with truncation to define the kernel function,  and subsequently model the nonlocal diffusion process through the bond-based formulation.  We successfully establish the well-posedness of  the proposed model along with deriving some of its properties on maximum principle and mass conservation.
 Furthermore, an efficient linear collocation scheme is designed for numerical solution of our model.  Comprehensive experiments in two and three dimensions are  conducted to showcase  application of the  proposed nonlocal model to both isotropic and anisotropic diffusion problems  and to demonstrate numerical accuracy  and effective asymptotic compatibility of the  proposed collocation scheme. 
 \end{abstract}

\begin{keyword}
Nonlocal model, anisotropic diffusion,  bond-based model,  Gaussian function, collocation scheme, asymptotic compatibility
\end{keyword}

\end{frontmatter}

\section{Introduction}

Isotropic and anisotropic diffusion phenomena have received much attention due to their broad applications, such as  inertial confinement fusion (ICF) \cite{LaLi2016,ReLi2017}, magnetic confinement fusion (MCF) \cite{Voge2011}, skyrmion diffusion in magnetism \cite{KeWe2021}, image processing \cite{PeMa1990,YoXu1996}, gas diffusion in fractal porous media \cite{MaCh2014}, and fluid distribution in fiber-reinforced composites \cite{GaFa2018}. These diffusion processes are usually described by partial differential equations (PDEs), for examples,  Fick’s law \cite{Vazq2006}, Darcy’s law \cite{Whit1986} and Fourier’s law \cite{GaDe1997}. The study of these equations is of great interest in numerous scientific  fields due to the importance of diffusion processes in physical systems.  The  linear diffusion equation with matrix-valued coefficients in non-divergence form  is  given by
\begin{equation}\label{locmod}
-\mathcal{L}u(\x):=-\sum^d_{i,j=1} a^{i,j}(\x) \frac{\partial^2 u}{\partial x_i \partial x_j}\left(\x\right)=f(\x),\quad \x \in \Omega,
\end{equation}
where $\Omega\in {\mathbb R}^d$ is a bounded Lipschitz domain,  $u(\x) :\Omega\rightarrow {\mathbb R}$ is the unknown function, and  the coefficient matrix $\mathbf{A}(\x)=\left(a^{i,j}(\boldsymbol{x})\right)_{i, j=1}^d $ is assumed to be elliptic, i.e., symmetric and positive definite and 
\begin{equation}\label{ellip}
\lambda|\boldsymbol{\xi}|^2 \leq \boldsymbol{\xi}^T \mathbf{A}(\boldsymbol{x}) \boldsymbol{\xi} \leq \Lambda|\boldsymbol{\xi}|^2, \quad \forall\, \boldsymbol{\xi} \in \mathbb{R}^d, \boldsymbol{x} \in \Omega,
\end{equation}
where $\lambda$ and $\Lambda$ are some positive constants. The diffusion process is said to be {\em isotropic} if $ \mathbf{A}(\boldsymbol{x})=a(\x){\bf I}$ where $a(\x)$ is a scalar-valued coefficient function and   ${\bf I}$ denotes the identity matrix, and otherwise it is said to be {\em anisotropic} . 
For the well-posedness of the problem  \eqref{locmod}, we refer to the comprehensive analysis and results presented in   \cite{GiTr1998}. For simplicity, we assume $a^{i,j}(\x)\in C^{0,1}(\Omega)$ and $f\in L^2(\Omega)$   in this paper, which guarantees the existence and uniqueness of a solution for  \eqref{locmod} under the Dirichlet boundary condition.
This type of equations have  been used in many applications, {such as stochastic optimal control and finance \cite{FlSo2006}, the optimal transportation problem and geometry \cite{CaGu1997,LeMi2017,TrWang2008}, the linearized fully nonlinear problems \cite{BrGu2011,Ne2013}.}
Since analytic solutions  are  usually unavailable in practice, many numerical methods  have been developed to solve the equation \eqref{locmod}, including the finite volume (FV) schemes \cite{ShYu2016,WaSh2022,YuYu2020}, the finite volume element (FVE) method \cite{GaYu2020,LvLi2012,ShYu2018} and so on.

It is known that classical diffusion models described by PDEs  often cannot provide a proper description of diffusion through heterogeneous materials, and they are limited in describing anomalous diffusion that does not obey Fick's law. Nonlocal diffusion model is an alternative to the PDE-based diffusion equation, which is based on Silling’s reformulation of the theory of elasticity for solid mechanic \cite{Sill2000,SiLe2010}. Bond-based models (i.e., unweighted nonlocal models) and state-based models (i.e., weighted nonlocal models) are the two major types of nonlocal diffusion models  \cite{Du2019, DuGu2013, DuZh2011, GuLe2010}. the former ones are  independent on the influence of other points in their domain of interaction and the latter ones are opposite. 

In the past two decades 
nonlocal diffusion models based on integral equations have gained extensive attention \cite{DuGu2012,DuJu2013,LuYa2022}.
The literature contains advanced engineering applications of nonlocal diffusion equations. For instance, some studies have explored the peridynamic formulation for heat transfer, including one-dimensional problems with different boundary conditions  \cite{BoDu2010}. Other studies have introduced multidimensional peridynamic formulations for transient heat transfer  \cite{BoDu2012}, and refined bond-based peridynamic approaches for thermo-mechanical coupling problems \cite{ChGu2021}. State-based peridynamic heat conduction equations and their applications were also investigated, with examples including the use of peridynamic differential operators for steady-state heat conduction analysis in plates with insulated cracks  \cite{Dord2019} and the nonlocal discrete model based on the lattice particle method for modeling anisotropic heat conduction \cite{ChLi2022}.
Nonlocal diffusion models also have been applied in fluid transport in porous media and corrosion. For example, state-based peridynamic formulations were developed to simulate convective transport of single-phase flow and transient moisture flow through heterogeneous and anisotropic porous media, as in \cite{KaFo2014} and \cite{JaMo2015}, respectively. Nonlocal fluid transport models were proposed to capture the non-local transport effects and non-linear mechanical behaviors of fluids in heterogeneous saturated porous media \cite{SuYu2022}. In the area of corrosion, a variety of coupled mechano-chemical peridynamic models were developed to describe stress-assisted corrosion and stress corrosion cracking; specific examples include models introduced in \cite{DuGu2012,JaCh2019,ZhJa2021}, which presents a coupled peridynamic corrosion-fracture model.

A stated-based nonlocal model was successfully proposed and analyzed to simulate the anisotropic diffusion in \cite{DeGu2021}, but it is theoretically not necessary to use the state-based formulation for modeling anisotropic diffusion. In \cite{DuGu2013}, a connection between the bond-based nonlocal model and the general diffusion process was established: the bond-based nonlocal diffusion operator based on the radial-type kernel function $\gamma(|\mathbf{z}|)$ with a compact support $B_{\delta }({\bf 0})$ converges to $-\nabla \cdot (\mathbf{A}\, \nabla)$ as  the horizon parameter $\delta \rightarrow 0$ under suitable conditions on  $\gamma(|\mathbf{z}|)$. However, to the best of our knowledge, such kernel function $\gamma(|\mathbf{z}|)$  has not been found in the literature when $\mathbf{A}$ is  anisotropic.
In this paper, we propose a novel bond-based nonlocal model which can simulate both isotropic and anisotropic diffusion processes.  Our approach
incorporates the coefficient matrix into a covariance matrix and then forms the  kernel function based on the multivariate Gaussian function. 
To ensure computational efficiency in practice, we further appropriately truncate the influence region of the kernel function  by considering its rapidly decaying nature.
We  establish the well-posedness of the proposed model under reasonable assumption on the coefficient matrix, and derive its maximum principle and when $\mathbf{A}$ is a constant matrix, the mass conservation. An efficient linear collocation discretization scheme is also proposed for numerical solution of our model, which is numerically shown to converge  exponentially  on uniform rectangular grids and be second-order accurate on  non-uniform grids. 
We further demonstrate through extensive experiments that the proposed collocation scheme  is effectively asymptotically compatible for solving  our nonlocal diffusion model, while the classic quadrature-based finite difference scheme \cite{DuTao2019,DuJuTi2017} fails  in the case of anisotropic coefficients.   

The rest of the paper is organized as follows. Section 2 briefly reviews some existing nonlocal  models for isotropic and anisotropic diffusion. Section 3 introduces the new bond-based nonlocal  diffusion model with matrix-valued coefficients and discusses the truncation of the influence region of the kernel function for its implementation in practice. Section 4 establishes the well-posedness of   the proposed model and derive some of its properties. Section 5 presents the linear collocation discretization scheme for numerical solution of our model. Section 6 provides extensive numerical experiments in two and three dimensions  to illustrate application of our model to various  isotropic and anisotropic diffusion problems  and to demonstrate numerical accuracy and effective asymptotic compatibility of the proposed collocation scheme. Some conclusions are finally given in Section 7.

\section{Review of existing nonlocal diffusion models}

\subsection{Nonlocal modeling for isotropic diffusion} 
Within the framework of nonlocal vector calculus \cite{DuGu2013}, given $\boldsymbol{\nu}(\x , \y ):\mathbb{R}^d \times \mathbb{R}^d \rightarrow \mathbb{R}^d$, the {\em unweighted nonlocal  divergence operator} $\mathcal{D}$ acting on $\boldsymbol{\nu}$ is defined as
\begin{equation}
\mathcal{D}  \boldsymbol{\nu}(\x):=\int_{\mathbb{R}^d}(\boldsymbol{\nu}(\x, \y)+\boldsymbol{\nu}(\y , \x )) \cdot \boldsymbol{\alpha}(\x, \y) \,d \y
\end{equation}
where $ \boldsymbol{\alpha}(\x , \y ): \mathbb{R}^d\times \mathbb{R}^d \rightarrow \mathbb{R}^d$ is a pre-determined antisymmetric vector, i.e. $\boldsymbol{\alpha}(\x , \y )=-\boldsymbol{\alpha}(\y , \x )$.  For any $u(\x):\mathbb{R}^d\rightarrow \mathbb{R}$, the unweighted adjoint (i.e., nonlocal gradient) operator $\mathcal{D}^{*}$ corresponding to $\mathcal{D}$ is then defined as
\begin{equation}
\mathcal{D}^{*} u(\x, \y):=(u(\y)-u(\x)) \boldsymbol{\alpha}(\x, \y).
\end{equation}
Let ${\Theta}(\x,\y):\mathbb{R}^d \times \mathbb{R}^d \rightarrow \mathbb{R}^{d\times d}$ be a  symmetric positive definite matrix-valued function. Under the above definitions of $\mathcal{D}$ and $\mathcal{D}^{*}$, we have
\begin{equation}\label{infl_conv}
\mathcal{D}(\Theta\cdot\mathcal{D}^{*}u)(\x)= 2\int_{\mathbb{R}^d}(u(\y)-u(\x)) \alpha(\x, \y) \cdot({\Theta}(\x, \y) \cdot \boldsymbol{\alpha}(\x, \y)) \,d \y.
\end{equation}
If we define the kernel function $\gamma(\x,\y) = 2\alpha \cdot(\Theta\cdot \alpha)$, which is clearly  nonnegative and symmetric, then (\ref{infl_conv}) also can be written as
\begin{equation}\label{gamm}
\mathcal{D}(\Theta\cdot\mathcal{D}^{*}u)(\x) =\int_{\mathbb{R}^d} (u(\y)-u(\x)){\gamma}(\x,\y) \,d \y.
\end{equation}
Let us take 
$$\boldsymbol{\alpha}(\x,\y)=\frac{\y-\x}{\|\y-\x\|},\quad\Theta(\x,\y)=a(\x)\omega(\x,\y)\mathbf{I},$$
where $a(\x): \mathbb{R}^d\rightarrow  \mathbb{R}$ is the scalar-valued diffusion coefficient  and  $\omega(\x,\y) = \omega(\| \y- \x \|):  \mathbb{R}^d\times  \mathbb{R}^d\rightarrow  \mathbb{R}$ is the influence function with a compact support $B_{\delta}(\x)=\{\y: \|\y-\x\|\leq \delta \}$ (the horizon parameter $\delta>0$). Then we have
 a bond-based nonlocal isotropic diffusion model as follows:
\begin{eqnarray}\label{bond_aniso}
\mathcal{D}(\Theta\cdot\mathcal{D}^{*}u)(\x)&=&\int_{\mathbb{R}^d}(u(\y)-u(\x))\frac{(\y-\x)^T}{\|\y-\x\|} \cdot (a(\x)\mathbf{I}+a(\y)\mathbf{I})\cdot \frac{\y-\x}{\|\y-\x\|} \omega(\x,\y)\,d\y \nonumber\\
&=&\int_{B_{\delta}(\x)}(u(\y)-u(\x)) \gamma(\x,\y) \,d\y. 
\end{eqnarray}
where $\gamma(\x,\y) = (a(\x)+a(\y)) \omega(\x,\y)$.
It is shown \cite{DuZh2011} that  as $\delta$ goes to 0, the nonlocal diffusion operator, $\mathcal{D}(\Theta\cdot\mathcal{D}^{*}u)$, converge to the classic isotropic diffusion operator in divergence form, $\nabla\cdot (a(\x) \nabla {u}$).

\subsection{Nonlocal modeling  for anisotropic diffusion}
A $\omega$-weighted nonlocal divergence operator  \cite{DuGu2013} acting on $\boldsymbol{\mu}(\x):\mathbb{R}^{d}\rightarrow\mathbb{R}^d$ is defined as
\begin{eqnarray}
\mathcal{D}_{\omega}\boldsymbol{\mu}(\x):=\mathcal{D}(\omega\boldsymbol{\mu})(\x).
\end{eqnarray}
The  $\omega$-weighted nonlocal  gradient operator $\mathcal{D}_{\omega}^{*}$ corresponding to $\mathcal{D}_{\omega}$ is correspondingly defined as
\begin{eqnarray}
\mathcal{D}_{\omega}^{*}u(\x):=\int_{\mathbb{R}^{d}}\omega(\x,\y)\mathcal{D}^{*}u(\x,\y) \,d \y. 
\end{eqnarray}
Then a state-based  nonlocal anisotropic diffusion model  can be defined as follows:
\begin{eqnarray}\label{stat_aniso}
\mathcal{D}_{\omega}(\mathbf{A}\cdot\mathcal{D}^{*}_{\omega}u)(\x)&=&\int_{\mathbb{R}^d} \big[\mathbf{A}(\x) \int_{\mathbb{R}^d} \omega(\x, \z) \mathcal{D}^* u(\x, \z) d \z \nonumber\\
&&\quad\quad+\mathbf{A} (\y)\int_{\mathbb{R}^d} \omega(\boldsymbol{\y}, \boldsymbol{\z}) \mathcal{D}^* u(\y, \boldsymbol{\z}) d \boldsymbol{\z})\big] \cdot \alpha(\boldsymbol{\x}, \boldsymbol{\y})\omega(\x, \y) d\boldsymbol{\y}.
\end{eqnarray} 
Note that the state-based model includes a two-fold integral while the bond-based one doesn't, thus
the bond-based one is usually more efficient in terms of computational cost.
With Taylor expansion, one can show that as $\delta$ goes to 0, the  nonlocal diffusion operator, $\mathcal{D}_{\omega}(\mathbf{A}(\x)\cdot\mathcal{D}^{*}_{\omega}u)$, converges to the classic  anisotropic diffusion operator in divergence form, $\nabla\cdot (\mathbf{A}(\x) \nabla{u}$).
It has been pointed out that under the suitable conditions on the kernel function $\gamma(\x,\y)$, the bond-based nonlocal model \eqref{gamm} also can be employed to simulate anisotropic diffusion. Assume that the kernel function is radially symmetric (i.e., $\gamma(\x,\y)=\gamma(|\y-\x|)$)  and  satisfies
\begin{equation}
a^{i,j}=\lim_{\delta\rightarrow 0}\int_{B_{\delta}(0)} \gamma(|\mathbf{z}|) z_i z_j d \mathbf{z}, \quad \text { for } \quad i, j=1,2, \ldots, d,
\end{equation}
then the nonlocal operator \eqref{gamm} converges to the classic diffusion operator  $\nabla\cdot (\mathbf{A} \nabla {u}$) as $\delta$ goes to 0 \cite{DuGu2013}.
Unfortunately, such a suitable kernel function has not yet been successfully constructed (i.e., appropriate choice of $\boldsymbol{\alpha}$ and $\Theta$) in the literature.

\section{A novel bond-based  nonlocal diffusion model with matrix-valued coefficients}
In this section, we will first present the  bond-based  nonlocal diffusion model and then establish its well-posedness and some of its properties.

\subsection{A nonlocal operator based on the Gaussian-type kernel function}

Let $p(\z,\boldsymbol{\mu},\boldsymbol{\Sigma})$ be the probability density function of  the $d$-dimensional multivariate  Gaussian (or normal) distribution with  expectation vector $\boldsymbol{\mu}$ and covariance matrix $\boldsymbol{\Sigma}$, which is defined by
\begin{equation}
p(\z,\boldsymbol{\mu},\Sigma) =\frac{1}{ \sqrt{{(2\pi)}^{d} \left|\boldsymbol{\Sigma}\right|}} \text{exp}\left(-\frac{1}{2} (\z-\boldsymbol{\mu})^{T} \boldsymbol{\Sigma}^{-1} (\z-\boldsymbol{\mu})\right).
\end{equation}
We first propose and study a  nonlocal diffusion operator with the coefficient matrix $\mathbf{A}$ defined as follows:
\begin{equation}\label{expnon}
\mathcal{L}_{\delta} u(\x):=\int_{\mathbb{R}^d} (u(\y)-u(\x))\gamma(\x,\y) \,d \y,\quad\,\x\in \mathbb{R}^d,
\end{equation}
where the kernel function ${\gamma(\x,\y)}$  is defined as
\begin{equation}\label{ndf}
{\gamma(\x,\y)}=\frac{2}{\delta^{2}} p(\y-\x,{\bf 0},\delta^{2} \mathbf{A} (\x)).
\end{equation} 

It is clear that ${\gamma(\x,\y)}>0$ and has an unbounded support region for any $\x\in\mathbb{R}^d$.
In stead of acting directly as the horizon's radius in the traditional bond-based nonlocal models, here the parameter $\delta$  is  used in constructing the covariance matrix. 
Note that as $\delta$ gets smaller, the kernel function $\gamma(\x,\y)$ would get  more singular, as illustrated in Figure \ref{delta} for some two dimensional cases. Another important difference with traditional nonlocal models is that the  coefficient matrix $\mathbf{A}(\x)$ is directly incorporated into and impact the kernel function, as illustrated in Figure \ref{coemat}. Although the influence region of ${\gamma(\x,\y)}$ for a point $\x$ is the whole space $\x\in\mathbb{R}^d$ theoretically, but the magnitude of the interaction between $\x$ and  other point $\y$ quickly shrinks to 0 when $\x$ and $\y$ get aways with each other. 

Let us denote the domain in which the nonlocal operator is applied by $\Omega_s$, and then the  nonlocal diffusion problem associated with the operator 
$\mathcal{L}_{\delta}$ under the Dirichlet-type boundary condition is  given by:
\begin{eqnarray} \label{dbc-org}
\left\{\begin{aligned}
 -\mathcal{L}_{\delta} u(\x ) &=f(\x),& &\quad {\x} \in \Omega_s, \\
u(\x) & =g(\x), & & \quad{\x} \in \mathbb{R}^d\setminus\Omega_s.
\end{aligned}\right.
\end{eqnarray}

\begin{figure}[!ht]
\centering{
\includegraphics[width=.48\textwidth]{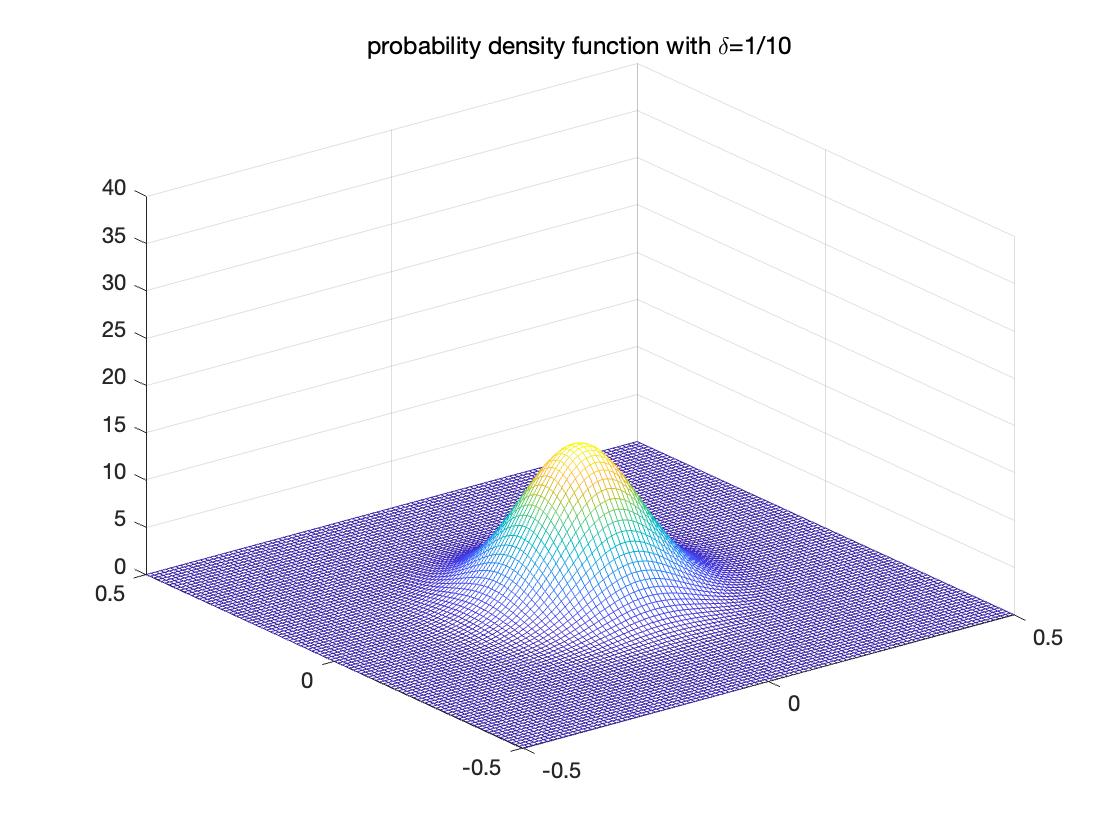}
\includegraphics[width=.48\textwidth]{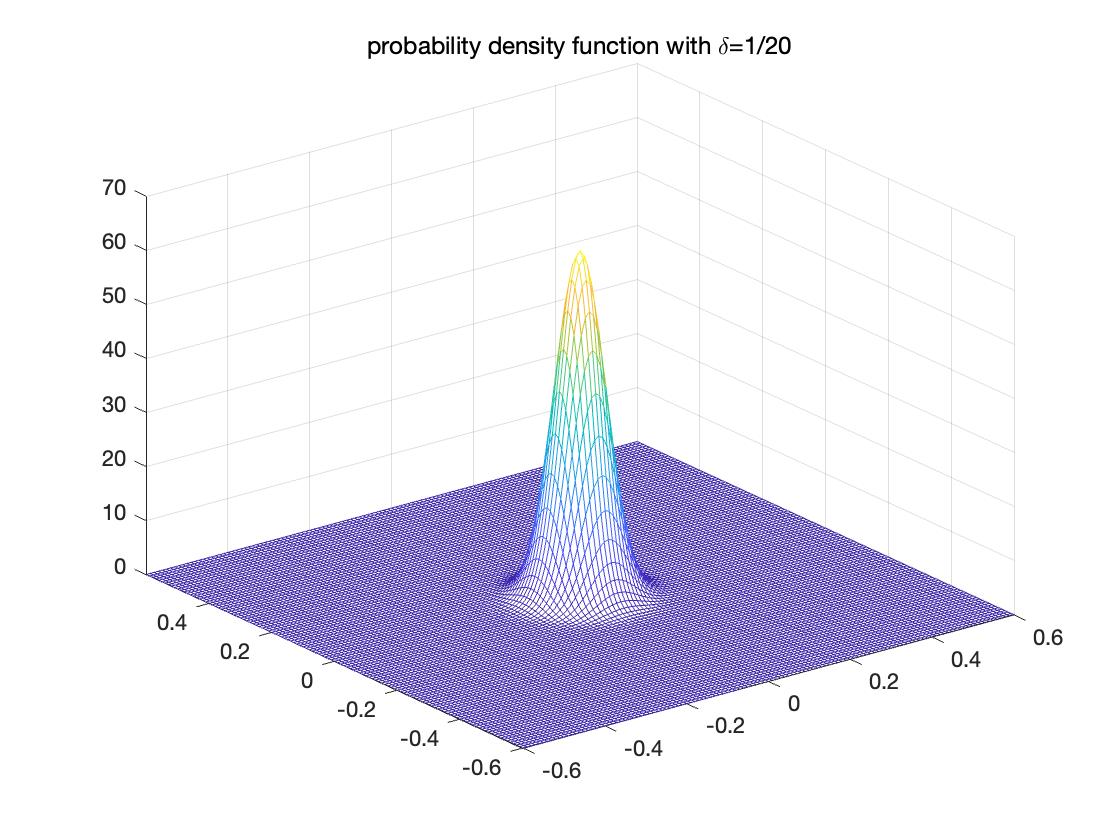}}
\caption{\em Plots of the  Gaussian function $p(\x,{\bf 0},\delta^{2} \mathbf{A})$  (in two dimensions) with  the isotropic constant coefficient matrix $\mathbf{A}=[1,0; 0,1]$ and  $\delta=1/10$ (left) and $\delta=1/20$ (right), respectively.  A smaller $\delta$ results in contours that are more elongated and concentrated around the center, while a larger $\delta$ leads to contours that are more spread out and diffuse.}
\label{delta}
\end{figure}

\begin{figure}[!ht]
\centering{
\includegraphics[width=.48\textwidth]{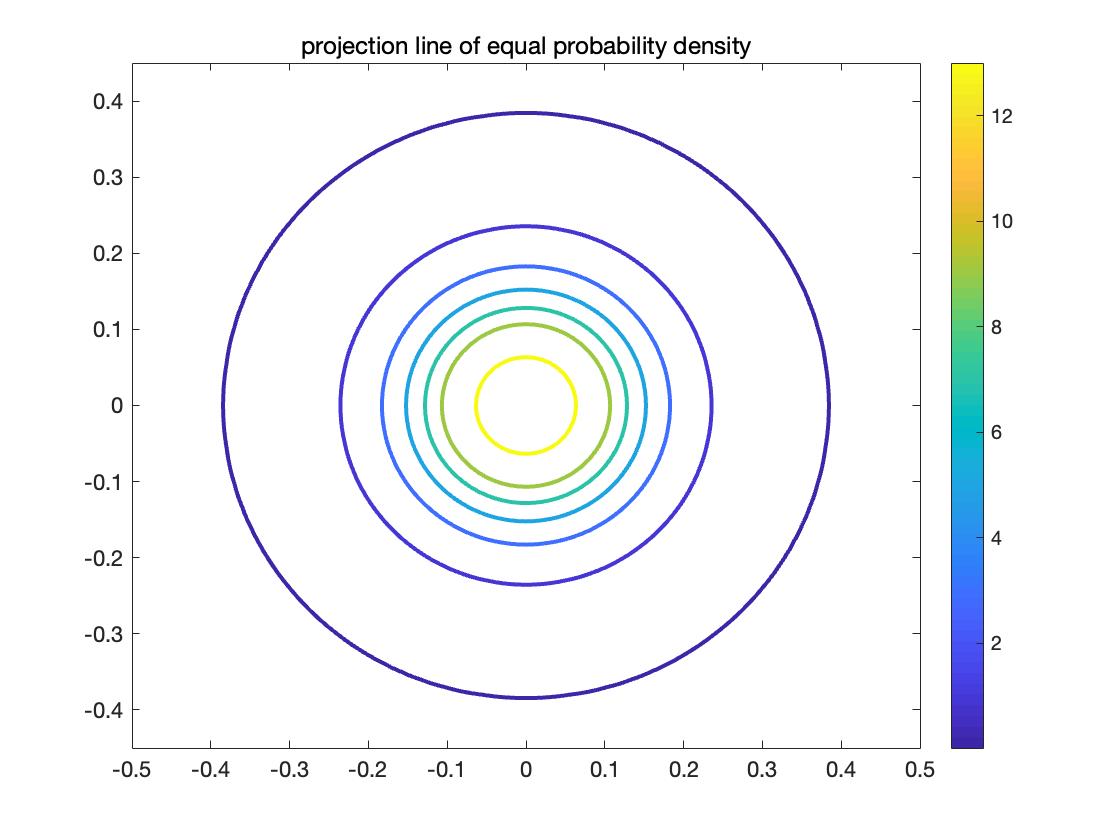}
\includegraphics[width=.48\textwidth]{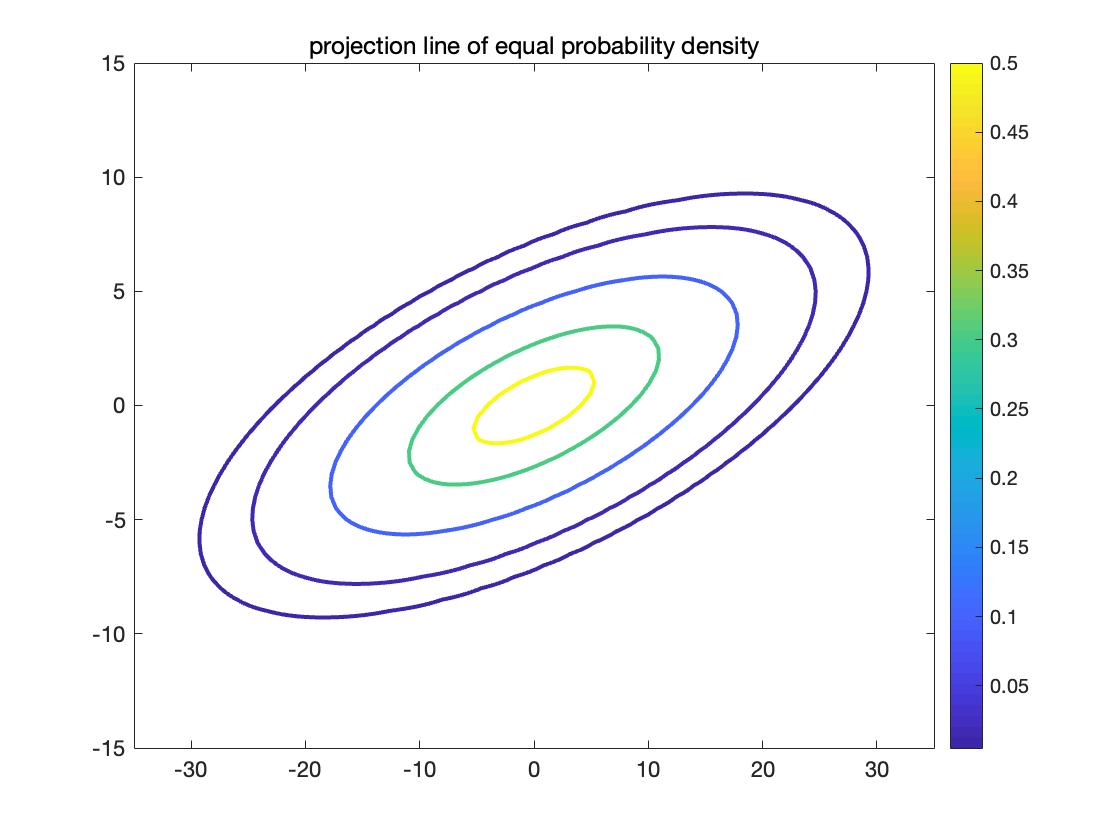}}
\caption{\em When the diffusion coefficient matrix $\mathbf{A} =[1,0;0,1]$, the probability density function is isotropic and thus the projection of its iso-density contours onto the coordinate plane takes the form of a circle (left).
When the diffusion coefficient matrix $\mathbf{A}=[10,2;2,1]$, the probability density function is anisotropic and subjected to a combined stretching and rotation, resulting in the rotated image of an elliptical shape upon projection onto the coordinate plane (right).}
\label{coemat}
\end{figure}

\subsection{Convergence to the local diffusion operator in non-divergence form}

Suppose the solution $u$ is sufficiently smooth. By applying Taylor expansion at $\x\in\Omega_s$, we obtain any $y\in\mathbb{R}^d$,
\begin{equation} \label{mat_h}
\begin{aligned}
u(\y)-u(\x)=\;&\nabla u(\x)^T (\y-\x)+\frac{1}{2!}(\y-\x)^{T}\mathbf{H}_u(\x) (\y-\x)\\
&+\sum_{p, q, r=1}^d \frac{1}{3 !} \frac{\partial^3 u\left(\x \right)}{\partial x_p \partial x_q \partial x_r} (\y-\x)_p (\y-\x)_q (\y-\x)_r+\cdots,
\end{aligned}
\end{equation}
where the matrix $\mathbf{H}_u$ is the  Hessians of  $u$. According to the definition of moments for the multivariate Gaussian distribution \cite{BlHw2019}, it is easy to see that the integrals of the first and third terms of the right-hand side of \eqref{mat_h} are 0, i.e.,
\begin{equation}\label{111}
\int_{\mathbb{R}^d}{\nabla u(\x)^T (\y-\x)\gamma(\x,\y)}\,d \y=0,
\end{equation}
and
\begin{equation}\label{222}
\int_{\mathbb{R}^d}{\sum_{p, q, r=1}^d \frac{1}{3 !} \frac{\partial^3 u\left(\x \right)}{\partial x_p \partial x_q \partial x_r} (\y-\x)_p (\y-\x)_q (\y-\x)_r\gamma(\x,\y)}d\y=0.
\end{equation}
For the integral of the second term of the right-hand side of \eqref{mat_h}, we have the following equality:
\begin{equation}\label{333}
\dfrac{1}{2}\int_{\mathbb{R}^d} {(\y-\x)_i(\y-\x)_j}\gamma(\x,\y)d\y=a^{i,j}(\x),\quad  1\leq i, j \leq d.
\end{equation}
Finally, the nonlocal operator defined in \eqref{expnon} can be written as:
\begin{equation}\label{modelerr}
\mathcal{L}_{\delta} u(\x)= \sum_{i,j=1}^{d}a^{i,j}(\x)\frac{\partial^2 u}{\partial x_i \partial x_j}\left(\x\right)+O(\delta^2).
\end{equation}
Hence, as the horizon parameter ${\delta}\rightarrow 0$ , the nonlocal operator ${\mathcal{L}_{\delta}}$ converges to the local diffusion operator in non-divergence form defined in \eqref{locmod}, i.e.,
$$\mathcal{L}_{\delta} u(\x) \rightarrow \mathcal{L} u(\x) :=\sum_{i,j=1}^{d}a^{i,j}(\x)\frac{\partial^2 u}{\partial x_i \partial x_j}\left(\x\right).$$
Furthermore, the convergence is expected to be quadratically with respect to $\delta$.
Thus the partial differential equation problem, as the local counterpart of the nonlocal diffusion problem {\eqref{dbc-org}}, is given by
\begin{eqnarray} \label{dbc-pde}
\left\{\begin{aligned}
 -\mathcal{L} u(\x ) &=f(\x),& &\quad {\x} \in \Omega_s, \\
u(\x) & =g(\x), & &\quad {\x} \in \partial\Omega_s,
\end{aligned}\right.
\end{eqnarray}
In the case of  constant  coefficient matrix $\mathbf{A}$, we further have
$\mathcal{L}_{\delta} u\rightarrow  \nabla \cdot(\mathbf{A}  \nabla) u$ since $\frac{\partial \mathbf{A}}{\partial x_i}\equiv {\bf 0}$.

\subsection{Truncation of the influence region for the kernel function}

The nonlocal model \eqref{expnon} utilizes a kernel function $\gamma(\x,\y)$ for $\mathcal{L}_{\delta}$ defined over an unbounded area, which  presents computational challenges for practical implementation. However, due to the rapid (exponentially) decay of the kernel function $\gamma$ when $\y$ gets away from the given point $\x\in\Omega_s$, it is feasible to truncate the influence region of $\gamma(\x,\y)$ to ensure computational efficiency in practice.
By choosing a suitable cut-off distance, one can effectively limit the computational domain to a finite region with just negligible  loss of the model  accuracy. This way allows for the application of the proposed nonlocal model in a more practical manner, while still capturing the essential physics of the system under consideration.

As is known that with any fixed $\x$, if $\y-\x$ follows the $d$-dimensional mean-zero Gaussian distribution with covariance matrix 
$\delta^{2} \mathbf{A} (\x)$, it holds that $(\y-\x)^{T} \mathbf{A(\x)}^{-1} (\y-\x)/\delta^2$ obey the chi-square distribution $\chi^2(d)$ {\cite{BlHw2019}}. If we take the influence region of  the kernel function at $\x$ to just include all $\y$ such that
$(\y-\x)^{T} \mathbf{A(\x)}^{-1} (\y-\x) \leq \delta^2\chi_\alpha^2(d),$ where {$0<\alpha\ll 1$} is a predetermined constant parameter and $\chi^2_{\alpha}(d)$ denotes the $(1-\alpha)$ quantile of the chi-square distribution, then we have
\begin{equation}\label{3}
\underset{(\y-\x)^{T} \mathbf{A(\x)}^{-1} (\y-\x) \leq \delta^2\chi_\alpha^2(d)} {\idotsint}p(\y-\x,{\bf 0},\delta^2\mathbf{A(\x)} )\; d\y=1-\alpha.
\end{equation}
We will take $$B_{\delta,{\bf A},\alpha}(\x) = \{\y\;|\; (\y-\x)^{T} \mathbf{A}(\x)^{-1} (\y-\x)\leq \delta^2\chi_\alpha^2(d)\}$$ as the truncated influence region for $\x$ with $\alpha$ selected very close to 0. 
Figure \ref{tuxing} illustrates the projection of iso-density contour of the truncated influence region onto the coordinate plane for different coefficient matrices in two dimensions,
where $\chi_\alpha^2(2)=36$, that is $\alpha\approx 1.52\times10^{-8}$.  For the identity matrix, the contour takes on a circular shape, as shown on the left-hand side of the figure.
For other matrices, such as those involving stretching and rotation, the contours may appear as rotated ellipses. 
Overall, the shape of the iso-density contours provides insight into the nature of the underlying probability density function for the kernel function and the effect of the transformation matrix on its properties. Thus, we correspondingly define the  truncated kernel function $\gamma_{\alpha}$ as: for any $\x,\y\in\mathbb{R}^d$,
\begin{equation}\label{newker}
\gamma_{\alpha}(\x,\y) =\left\{
\begin{aligned}
\gamma(\x,\y),& \qquad \y\in B_{\delta,{\bf A},\alpha}(\x),\\
0,&\qquad {\rm otherwise.}
\end{aligned} \right.
\end{equation} 
Note that the identities \eqref{111} and \eqref{222} still hold exactly  for the above truncated kernel function $\gamma_{\alpha}(\x,\y)$, but the equality \eqref{333} doesn't and its error depends on the choice of $\alpha$ and  converges to 0 as $\delta$ goes to 0. 

\begin{figure}[h]
\centering{
\subfigure[\label{fig:a}]{
	\begin{minipage}[t]{.22\linewidth}
		\centering
		\includegraphics[width=.60\linewidth]{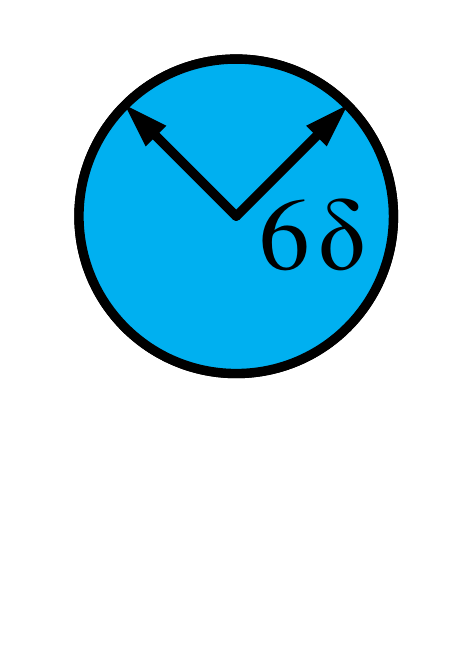}
	\end{minipage}
}
\hspace{-4mm}
\subfigure[\label{fig:b}]{
	\begin{minipage}[t]{.35\linewidth}
		\centering
		\includegraphics[width=1.1\linewidth]{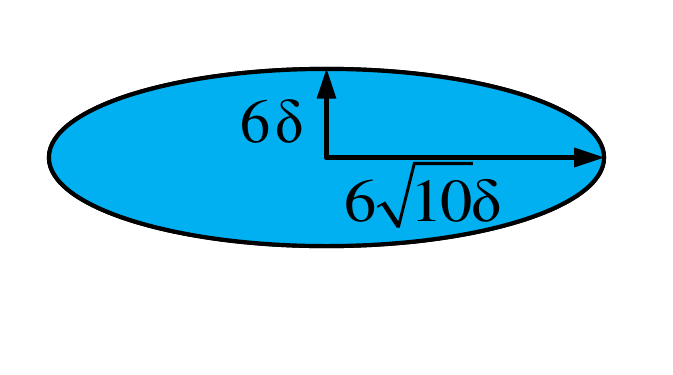}
	\end{minipage}
}
\hspace{5.5mm}
\subfigure[\label{fig:c}]{
	\begin{minipage}[t]{.32\linewidth}
		\centering
		\includegraphics[width=1.1\linewidth]{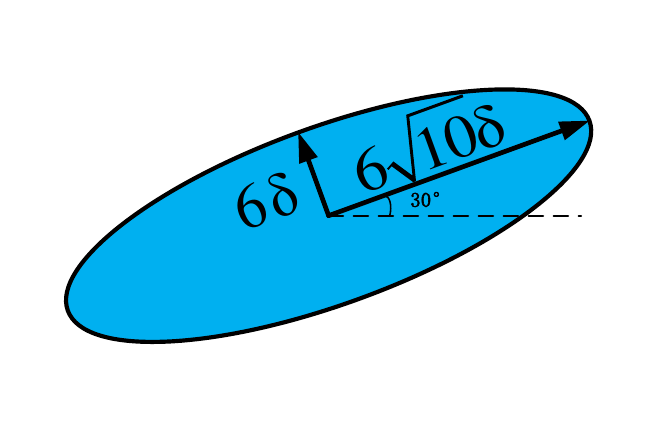}
	\end{minipage}
}
}
\caption{\em Illustration of the truncated influence region $B_{\delta,{\bf A},\alpha}({\bf 0})$  with $\chi_\alpha^2(2)=36$ {(i.e., $\alpha\approx 1.52\times10^{-8}$)}  in two dimensions. Left: the identity coefficient matrix $\mathbf{A}$=[1,0; 0,1], the truncated influence region  is a disk with radius $6\delta$; middle: an anisotropic coefficient matrix $\mathbf{A}$=[10, 0; 0,1], the truncated influence region  is an elliptic region whose semi-major axis is $6\sqrt{10} \delta$ and semi-minor axis is $6\delta$; right: $\mathbf{A}=[31/4,-9\sqrt{3}/4;-9\sqrt{3}/4,13/4]$, the truncated influence region of is a rotated (counterclockwisely $30^{\circ}$) elliptic region whose semi-major axis is $6\sqrt{10} \delta$ and semi-minor axis is $6\delta$.}
\label{pic3}
\label{tuxing}
\end{figure}

In the context of nonlocal operators,  it is common to impose constraints on certain interaction domain $\Omega_c$, which has a positive volume.  These volume constraints are the natural extensions of boundary conditions commonly used in differential equation problems.
By specifying appropriate volume constraints, one can effectively control the behavior of the nonlocal operator in $\Omega_s$, thereby ensuring the desired properties of the solution. For example, such constraints may include restrictions on the overall quality or energy of the system, or restrictions on the nonlocal kernel function support. In this way, volume constraints offer a powerful tool for designing and optimizing nonlocal operator models, and have wide-ranging applications in the study of physical systems and processes. Given the domain $\Omega_s \in \mathbb{R}^d$, we define the corresponding interaction domain  associated with the kernel function $\gamma_{\alpha}$  as
$$\Omega_c:=\left\{\y \in \mathbb{R}^d \setminus \Omega_s \;|\;  \exists\, \x\in \Omega_s,\;\y\in B_{\delta,{\bf A},\alpha}(\x)\right\}.$$
and the corresponding truncated nonlocal diffusion operator as
\begin{equation}\label{expnon-mod}
\mathcal{L}_{\delta,\alpha} u(\x):=\int_{\mathbb{R}^d} (u(\y)-u(\x))\gamma_{\alpha}(\x,\y) \,d \y,\quad\,\x\in \Omega_s.
\end{equation}
We remark that  the nonlocal diffusion operator $\mathcal{L}_{\delta,\alpha} u(\x )$ is an approximation of the original  nonlocal diffusion operator $ \mathcal{L}_{\delta} u(\x )$, and  $\mathcal{L}_{\delta,\alpha} u(\x )$ converges to  $\mathcal{L}_{\delta} u(\x )$ quickly as $\alpha$ goes to 0.
Finally, we obtain the  corresponding bond-based nonlocal diffusion problem under the volume constraint-based Dirichlet boundary condition as follow:
\begin{eqnarray} \label{dbc}
\left\{\begin{aligned}
 -\mathcal{L}_{\delta,\alpha} u(\x )&=f(\x),& &\quad {\x} \in \Omega_s, \\
u(\x) & =g(\x), & &\quad {\x} \in \Omega_c,
\end{aligned}\right.
\end{eqnarray}
where $\Omega=\Omega_c \cup\Omega_s$.

\section{Wellposedness and properties}

We now study the well-posedness of the bond-based nonlocal diffusion model \eqref{dbc}.
Let us define the constrained space 
$L_{n0}^{2}(\Omega)=\{u\in L^{2}(\Omega)\,|\, u=0 \text{ on }\Omega_c\}$ and  denote by $(\cdot,\cdot)$ the $L^2$ inner product. Assume  $f\in L^2(\Omega_s)$. Then a weak formulation of the proposed nonlocal diffusion  model \eqref{dbc} with the homogeneous boundary value $g|_{\Omega_c}=0$ is given as follows: find $u\in L_{n0}^2(\Omega)$ such that 

\begin{eqnarray}\label{weak_form_3}
{\bf B}(u,v)={\bf F}(v),\quad \forall\; v\in L_{n0}^2(\Omega)\,,
\end{eqnarray}
where $${\bf B}(u,v)=(-\mathcal{L}_{\delta,\alpha} u,v)=\displaystyle\int_{\Omega} \int_{\Omega} (u(\x)-u(\y))v(\x){\gamma_{\alpha}}(\x,\y)\, d \y d\x$$ and $$ {\bf F}(v)=\int_\Omega f(\x)v(\x)\,d\x.$$
Since $\gamma_{\alpha}(\x,\y)$ is a truncated, scaled probability density function as defined in \eqref{newker} and the coefficient matrix ${\bf A}$ is elliptic,  it is not hard to see that the kernel $\gamma_\alpha$ satisfies:  for any $\x\in\Omega$, 
\begin{equation}\label{kercondition1}
\begin{aligned}
	&\int_\Omega \gamma_{\alpha}(\x,\y)\,d\y\leq \frac{2}{\delta^2}, \quad\int_\Omega \gamma_{\alpha}(\y,\x)\,d\y\leq K_1(\delta),\\[4pt]
	&\int_{\Omega_c} \frac{\gamma_{\alpha}(\x,\y)+\gamma_{\alpha}(\y,\x)}{2}d\y\geq K_2(\delta),
\end{aligned}
\end{equation}
where  $K_1(\delta)$ and $K_2(\delta)$ are two positive constants depending on $\delta$.
Note that $\gamma_{\alpha}(\x,\y)$ isn't symmetric  except when ${\bf A}(\x)$ is a constant matrix, thus the bilinear operator ${\bf B}(u,v)$ is not symmetric in general.
We then obtain the following result on the well-posedness of the  nonlocal diffusion problem \eqref{dbc}. 

\begin{thm}\label{well-posedness}{\rm (Well-posedness)} 
Let the parameter $\delta>0$ be fixed. Assume that the kernel function $\gamma_{\alpha}(\x,\y)$ satisfies
	\begin{equation}\label{kercondition2}
	\int_{\Omega} \frac{\gamma_{\alpha}(\x,\y)-\gamma_{\alpha}(\y,\x)}{2}\,d\y\geq 0,\quad \forall\,\x\in\Omega.
	\end{equation}
	 Then there exists a unique solution $u\in L_{n0}^{2}(\Omega)$ to the nonlocal diffusion problem \eqref{dbc}. Furthermore, the solution satisfies the a priori estimate
	\begin{equation}\label{weak7} \|u\|_{L^2(\Omega)}\leq\frac{2}{K_2(\delta)} \|f\|_{L^2(\Omega)}.
	\end{equation}
\end{thm}
\begin{proof}
	First, it is easy to see that $\mathbf{F}$ is a bounded linear functional:
	\begin{equation}\label{weak71}
		|\mathbf{F}(v)| \leq\|f\|_{L^2(\Omega)}\|v\|_{L^2\left(\Omega_s\right)}.
		\end{equation}

	We now show that the  bilinear operator {\bf B}($\cdot,\cdot$)  is bounded on $L_{n0}^2(\Omega)\times L_{n0}^2(\Omega)$ by using a similar proof as that of Theorem 2 in \cite{DuJuTi2017}. For any $u,v \in L_{n0}^2(\Omega)$, it holds
	\begin{eqnarray}\label{weak_form_5}
		{\bf B}(u,v)=\int_{\Omega}\int_{\Omega} u(\x)v(\x)\bd{\gamma}_{\alpha}(\x,\y)d\y d\x-\int_{\Omega}\int_{\Omega} u(\y)v(\x)\bd{\gamma}_{\alpha}(\x,\y)d\y d\x:= I_1-I_2.
	\end{eqnarray}
Using the inequality in (\ref{kercondition1}) and the Cauchy-Schwartz inequality, the first term of the right-hand side of (\ref{weak_form_5}) satisfies
	\begin{equation*}
	\begin{aligned}
		|I_1|& =\Big|\int_{\Omega}u(\x)v(\x)\int_{\Omega} \bd{\gamma}_\alpha(\x,\y) \,d \y d\x \Big|
		\leq  \frac{2}{\delta^2}\|u\|_{L^2(\Omega)}\|v\|_{L^2(\Omega)},
\end{aligned}
\end{equation*}
and the second term of the right-hand side of (\ref{weak_form_5}) satisfies
	\begin{equation*}
	\begin{aligned}
		|I_2| &\leq \Big(\int_{\Omega}\int_{\Omega} u^2(\y)\bd{\gamma}_\alpha(\x,\y)\,d \y d\x\Big)^{\frac12}\Big(\int_{\Omega}\int_{\Omega} v^2(\x)\bd{\gamma}_\alpha(\x,\y)\,d \y d\x\Big)^{\frac12}\\[3pt]
		&\leq  \frac{\sqrt{2K_1(\delta)}}{\delta}\|u\|_{L^2(\Omega)}\|v\|_{L^2(\Omega)}.
		\end{aligned}
	\end{equation*}
Thus we obtain
\begin{equation}
	{\bf B}(u,v)\leq|I_1|+|I_2|=\left(\frac{2}{\delta^2}+\frac{\sqrt{2 K_1(\delta)}}{\delta}\right)\|u\|_{L^2(\Omega)}\|v\|_{L^2(\Omega)}.
\end{equation}

Next we show that the bilinear operator {\bf B}($\cdot,\cdot$)  is coercive on $L_{n0}^2(\Omega)$. Note that
\begin{equation}\label{weak4}
\begin{aligned}
	{\bf B}(u,u)=&\int_{\Omega}\int_{\Omega} \left(u(\x) -u(\y)\right)u(\x)\bd{\gamma}_\alpha(\x,\y)\,d \y d\x\\
	=&\int_{\Omega}\int_{\Omega} \left(u(\x) -u(\y)\right)u(\x)\frac{\bd{\gamma}_\alpha(\x,\y)+\bd{\gamma}_\alpha(\y,\x)}{2}\,d \y d\x\\
	& +\int_{\Omega}\int_{\Omega} \left(u(\x) -u(\y)\right)u(\x)\frac{\bd{\gamma}_\alpha(\x,\y)-\bd{\gamma}_\alpha(\y,\x)}{2}\,d \y d\x := J_1+J_2.
\end{aligned}
\end{equation}
It is easy to see that
$\frac{\bd{\gamma}_\alpha(\x,\y)+\bd{\gamma}_\alpha(\y,\x)}{2}$ is always symmetric although $\gamma_\alpha(\x,\y)$ isn't.  Then the first term in the right-hand side of (\ref{weak4}) satisfies
\begin{equation}\label{u0}
\begin{aligned}
      J_1  &=\frac12\int_{\Omega}\int_{\Omega} \left(u(\x) -u(\y)\right)^2\frac{\bd{\gamma}_\alpha(\x,\y)+\bd{\gamma}_\alpha(\y,\x)}{2}\,d \y d\x \\
	&\geq \frac12\int_{\Omega}\int_{\Omega_c} \left(u(\x) -u(\y)\right)^2\frac{\bd{\gamma}_\alpha(\x,\y)+\bd{\gamma}_\alpha(\y,\x)}{2}\,d \y d\x\\
	 & = \frac12\int_{\Omega}u^2(\x)\int_{\Omega_c} \frac{\bd{\gamma}_\alpha(\x,\y)+\bd{\gamma}_\alpha(\y,\x)}{2}\,d \y d\x\geq\frac{K_2(\delta)}{2}\|u\|^2_{L^2(\Omega)}.
\end{aligned}
\end{equation}
Using the assumption \eqref{kercondition2} and the equality 
\begin{eqnarray*}
	\int_{\Omega}\int_{\Omega} u(\x)u(\y)\frac{\bd{\gamma}_\alpha(\x,\y)-\bd{\gamma}_\alpha(\y,\x)}{2}d\y d\x=0,
	\end{eqnarray*}
we obtain for the second term in the right-hand side of (\ref{weak4})
\begin{equation}\label{weak6}
\begin{aligned}
J_2&=\int_{\Omega}\int_{\Omega} u^2(\x)\frac{\bd{\gamma}_\alpha(\x,\y)-\bd{\gamma}_\alpha(\y,\x)}{2}\,d\y d\x
	-\int_{\Omega}\int_{\Omega} u(\x)u(\y)\frac{\bd{\gamma}_\alpha(\x,\y)-\bd{\gamma}_\alpha(\y,\x)}{2}\,d\y d\x\\
	&=\int_{\Omega}u^2(\x)\int_{\Omega_c} \frac{\bd{\gamma}_\alpha(\x,\y)-\bd{\gamma}_\alpha(\y,\x)}{2}\,d \y d\x\geq 0.
\end{aligned}
\end{equation}
The combination of \eqref{u0} and \eqref{weak6} gives us
$${\bf B}(u,u)\geq\frac{K_2(\delta)}{2}\|u\|^2_{L^2(\Omega)}$$

Thus, by the Lax-Milgram theorem, there exists a unique weak solution $u\in L_{n0}^2(\Omega)$ for the nonlocal diffusion problem \eqref{dbc}. Furthermore, since
\begin{equation}
	\frac{K_2(\delta)}{2}\|u\|_{L^2(\Omega)}^2 \leq \boldsymbol{B}(u, u)=|\mathbf{F}(u)| \leq\|f\|_{L^2(\Omega)}\|u\|_{L^2(\Omega)}
	\end{equation}
we have the a priori estimate (\ref{weak7}).
	\end{proof}

\begin{rem}
In the case of $a^{i,j}(\x)\in C^2(\Omega)$, notice that the classic diffusion operator in non-divergence form also can be written as
\begin{equation}
-\mathcal{L}u(\x):=-\sum^d_{i,j=1} a^{i,j}(\x) \frac{\partial^2 u}{\partial x_i \partial x_j}\left(\x\right)=-\nabla\cdot({\mathbf A}(\x)\nabla u(\x))+{\bf b}(\x)\cdot \nabla u(\x),
\end{equation}
where ${\bf b}(\x) = (b^i(\x))$ with $b^i(\x) =  \sum^d_{j=1}
\frac{\partial a^{i,j}}{\partial x_j}(\x)$.
Thus the assumption \eqref{kercondition2}  can be regarded as an analogue  in the nonlocal sense to the well-known
condition for the convection velocity
\begin{equation}\label{condp}
\nabla \cdot {\bf b}(\x)\leq  0, \quad\x\in\Omega.
\end{equation}
\end{rem}

\begin{lem}\label{Maximum Principle}{\rm (Maximum principle)} If $\mathcal{L}_{\delta,\alpha} u>0$ in $\Omega_s$, then a maximum of $u$ is only attained in the interaction domain $\Omega_c$.
\end{lem}
\begin{proof}
Assume that $u$ attains a  maximum at $\x_0\in\Omega_s$, then we have
\begin{equation}
\mathcal{L}_{\delta,\alpha} u(\x_0)=\int_{\Omega} (u(\y)-u(\x_0)){\gamma_{\alpha}}(\x,\y) \,d \y.
\end{equation}
Since $u(\y)-u(\x_0)\leq0$, it is easy to verify that
\begin{equation}
\int_{\Omega}\overbrace{(u(\y)-u(\x_0))}^{\leq 0}\overbrace{
\frac{2}{{\delta^{2} \sqrt{{(2\pi)}^{n} \left|\delta^2{\bf A(\x_0)}\right|}}} \text{exp}\left(-\frac{1}{2} (\y-\x_0)^{T} ({\delta^2\bf A(\x_0)})^{-1} (\y-\x_0)\right)
}^{\geq0}\leq 0,
\end{equation}
which is a contradiction with the assumption of $\mathcal{L}_{\delta,\alpha} u(\x_0)>0$.
\end{proof}

Let us consider the case of  the coefficient matrix $\mathbf{A}(\x)$ being constant. In this case $\gamma_{\alpha}(\x,\y)=\gamma_{\alpha}(\y,\x)$, and we can further show that  the proposed nonlocal diffusion model \eqref{dbc}  satisfies global mass conservation. Integrating both sides of \eqref{dbc} on ${\Omega_s}$,  we have
\begin{equation}
\begin{aligned}
-\int_{\Omega_s} \mathcal{L}_{\delta,\alpha} u(\x) d \x=&\int_{\Omega_s} \int_{\Omega_s\cup\Omega_c}(u(\x)-u(\y)){\gamma_{\alpha}}(\x,\y)\,\,d \y d \x\\
=&\int_{\Omega_s} \int_{\Omega_s}(u(\x)-u(\y)){\gamma_{\alpha}}(\x,\y)\,d \y d \x
+\int_{\Omega_s} \int_{\Omega_c}(u(\x)-u(\y)){\gamma_{\alpha}}(\x,\y)\,d \y d \x.
\end{aligned}
\end{equation}
Interchanging the variable $\x$ and $\y$, we get
\begin{eqnarray*}
&&\int_{\Omega_s} \int_{\Omega_s}u(\x){\gamma_{\alpha}}(\x,\y)-u(\y){\gamma_{\alpha}}(\x,\y)\,\,d \y d \x  \nonumber\\
&&\qquad=\dfrac12\Big(\int_{\Omega_s} \int_{\Omega_s}(u(\x)-u(\y)){\gamma_{\alpha}}(\x,\y)+(u(\y)-u(\x)){\gamma_{\alpha}}(\y,\x)\,d \y d \x\Big)=0,
\end{eqnarray*}
which means the flux from $\Omega_c$ into itself is zero. Thus we can get
\begin{equation}
-\int_{\Omega_s} \mathcal{L}_{\delta,\alpha} u(\x) d \x	=\int_{\Omega_s} \int_{\Omega_c}(u(\x)-u(\y)){\gamma_{\alpha}}(\x,\y)\,\,d \y d \x,
\end{equation}
and   the {\em mass conservation} of the proposed nonlocal model is then obtained as
\begin{equation}
\int_{\Omega_s} \int_{\Omega_c}(u(\x)-u(\y)){\gamma_{\alpha}}(\x,\y)\,\,d \y d \x=\int_{\Omega_s} f(\x)\,d\x .
\end{equation}
Here, the left-hand side term is the nonlocal flux out of $\Omega$ into $\Omega_s$, which is a proxy for the interaction between $\Omega_c$ and $\Omega_s$, while
the right-hand side term is the source from $\Omega_c$, and putting together, they cancel out each other. Note that such mass conservation doesn't hold anymore if $\mathbf{A}(\x)$
isn't a constant matrix.

\section{Numerical discretization}

In this section we propose and analyze an efficient linear collocation scheme for discretizing the nonlocal diffusion problem \eqref{dbc}. We refer to \cite{DDGGTZ20}
for a comprehensive review of existing numerical methods for nonlocal models.

\subsection{A linear collocation discretization scheme}
Without loss of generality, assume a rectangular grid $\mathcal{T}_h$ of  the domain $\Omega$ is given. Note that the discretization scheme developed below 
also can be similarly generalized to triangular/tetrahedral meshes with linear interpolation.  Let us denote the  nodes  of $\mathcal{T}_h$ as $\{\x_1,\x_2,\cdots, \x_{N_s}\}\in\Omega_s$ (interior nodes) and $\{\x_{N_s+1},\x_{N_s+2},\cdots, \x_{N_s+N_c}\}\in\Omega_c$ (nonlocal
boundary nodes). Denote by $\phi_i(\x)$  the standard continuous piecewise bilinear (if $d=2$) or trilinear (if $d=3$) basis function  at the point $\x_i$ and by $S_i$ the support region for each $\phi_i(\x)$. 
At each interior node $\x_i$ for $i=1,2,\cdots,N_s$,  we have for the equation \eqref{dbc} that
\begin{equation} \label{cran_k}
-\mathcal{L}_{\delta,\alpha}u(\x_i)=\int_{B_{\delta,{\bf A},\alpha}(\x_i)}(u(\x_i)-u(\y))\gamma_{\alpha}(\x_i,\y)\,d\y,
\end{equation}
then we approximate it by
\begin{equation} \label{shrub}
-\mathcal{L}^h_{\delta,\alpha} u(\x_i):=\int_{B_{\delta,{\bf A},\alpha}(\x_i)} \mathcal{I}_h\left((u(\x_i)-u(\y))\gamma_{\alpha}(\x_i,\y)\right)\,d\y,
\end{equation}
where 
$\mathcal{I}_h(\cdot)$ represents the piecewise bilinear (or trilinear) interpolation operator associated with the grid ${\mathcal T}_h$.
Note  \eqref{shrub} can be further rewritten  as 
\begin{eqnarray}\label{detailed_Interpolate1}
-\mathcal{L}^h_{\delta,\alpha} u(\x_i) &=& \sum_{{\x_j\in\Omega}\,\&\,j\ne i}(u(\x_i)-u(\x_j)) \gamma_{\alpha}(\x_i,\x_j)\int_{B_{\delta,{\bf A},\alpha}(\x_i)}\phi_j(\y) \,d\y. 
\end{eqnarray} 

Finally, a linear collocation scheme for solving the nonlocal diffusion problem \eqref{dbc} can be given as follows: find $(u_h(\x_1),\cdots,u_h(\x_{N_s}))$ such that
\begin{eqnarray}\label{detailed_Interpolate_linear}
-\mathcal{L}^h_{\delta,\alpha} u_h(\x_i) = f(\x_i), \qquad i=1,\cdots,N_s,
\end{eqnarray}

Let us define 
\begin{equation}\label{stiff_matrix_coef}
b_{i,j}=
\begin{cases}
- \gamma_{\alpha}(\x_i,\x_j)\displaystyle\int_{B_{\delta,{\bf A},\alpha}(\x_i)}\phi_j(\y) \,d\y,& \mbox{if } j\neq i,\\
-{\sum_{{\x_j\in\Omega}\,\&\,j\ne i}}\,b_{i,j},  & \mbox{if } j= i,
\end{cases}
\end{equation}
for $1\leq i\leq N_s$ and $1\leq j\leq N_s+N_c$. 
Then the resulting linear system from the  collocation discretization scheme \eqref{detailed_Interpolate_linear} can be expressed by
\begin{equation}\label{linearsys}
\unl{B}_h\vec{u}_h= \vec{f},
\end{equation}
where $\vec{u}_h=(u_h(\x_1),\cdots,u_h(\x_{N_s}))^T$, $\vec{f}=(f_1,\cdots,f_{N_s})^T$
with $\unl{B}_h=(b_{i,j})_{N_s\times N_s}$ and 
\begin{equation}
f_i = f(\x_i)+\sum_{j=N_s+1}^{N_s+N_c} {g(\x_j)}b_{i,j}.
\end{equation}

\begin{thm}\label{stiff_matrix}  
The stiffness matrix ${B}_h$ defined in \eqref{stiff_matrix_coef} is a nonsingular  \unl{M}-matrix. Consequently, the linear system \eqref{linearsys} produced from the  collocation  discretization \eqref{detailed_Interpolate_linear} for the nonlocal diffusion problem \eqref{dbc}  is uniquely solvable, and $u_h$ always satisfies the following discrete maximum principle when $g=0$:  if $f\leq 0$ in $\Omega_s$, then $\max_{1\leq i\leq N_s}u_h(\x_i)\leq 0$, and if  $f\geq 0$ in $\Omega_s$, then $\min_{1\leq i\leq N_s}u_h(\x_i)\geq 0$.
\end{thm}
\begin{proof}
It is obvious that  for $1\leq i\leq N_s$ and $1\leq j\leq N_s+N_c$, we have $\gamma_{\alpha}(\x_i,\x_j)>0$ and $\int_{B_{\delta,{\bf A},\alpha}(\x_i)}\phi_j(\y) \,d\y>0$ if $i\ne j$ and  $B_{\delta,{\bf A},\alpha}(\x_i)\cap S_j\ne \emptyset$. 
Thus by (\ref{stiff_matrix_coef}), it holds that for any $j\ne i$, we have 
$b_{i,j}< 0$ if $B_{\delta,{\bf A},\alpha}(\x_i)\cap S_j \ne \emptyset$ and $b_{i,j}= 0$ if $B_{\delta,{\bf A},\alpha}(\x_i)\cap S_j = \emptyset$ (since  $\phi_j(\y)= 0$  for $\y\notin S_j$).
Consequently we have $b_{i,i}> 0$ for $1\leq i\leq N_s$. Furthermore, it is easy to see that 
$\sum_{j=1}^{N_s} b_{i,j}>0$  if $B_{\delta,{\bf A},\alpha}(\x_i)\cap\Omega_c\ne\emptyset$. 
Thus,  the stiffness matrix  ${B}_h$ is an $\unl{M}$-matrix, which means that ${B}_h^{-1}$ exists and does not have  negative entries. The unique solvability of the linear system \eqref{linearsys} and the discrete maximum principle then directly follows. 
\end{proof}

Theorem \ref{stiff_matrix}  guarantees  the numerical stability of the  proposed linear collocation discretization scheme \eqref{detailed_Interpolate_linear}.  
In the proposed model, the influence region $B_{\delta,{\bf A},\alpha}(\x_i)$ of each  node $\x_i$, is a rotated ellipse. The intersection between $B_{\delta,{\bf A},\alpha}(\x_i)$ and the support region $S_j$ of another node $\x_j$ may be irregular, especially when $\x_j$ is near the edge of $B_{\delta,{\bf A},\alpha}(\x_i)$. This irregular intersection could introduce quadrature errors and potentially impact the accuracy of the simulation. To address this issue, we expand the integral region from the irregular intersection $B_{\delta,{\bf A},\alpha}(\x_i) \cap S_j$ to the regular support domain $S_j$. This expansion allows for easy integration by aligning the integration area precisely with the grid.

\subsection{Discussion on convergence and asymptotic compatibility}
Let us  assume that the parameter $\alpha$ is small enough so that {\em the difference between  the original operator $\mathcal{L}_{\delta} $ \eqref{dbc-org}  and the truncated operator $\mathcal{L}_{\delta,\alpha}$ \eqref{expnon-mod} is negligible}. We first consider a rectangular grid $\mathcal{T}_h$ of $\Omega$ which is uniform along each space direction with respective grid sizes. If we extend the grid to whole space $\mathbb{R}^d$, then 
$$-\mathcal{L}^h_{\delta} u(\x_i) :=\int_{\mathbb{R}^d} \mathcal{I}_h\left((u(\x_i)-u(\y))\gamma(\x_i,\y)\right)\,d\y$$
in fact gives the  trapezoidal rule for evaluating the original nonlocal diffusion operator \eqref{expnon} (no truncation of the influence region)
\begin{equation*}
\begin{aligned}
-\mathcal{L}_{\delta} u(\x_i) &= \int_{\mathbb{R}^d} (u(\x_i)-u(\y))\gamma(\x_i,\y)\,d\y\\
&=\frac{2}{\delta^{2}}\int_{\mathbb{R}^d} (u(\x_i)-u(\y))\frac{1}{ \sqrt{{(2\pi)}^{d} \left|\delta^2\boldsymbol{\mathbf{A}(\x_i)}\right|}} \text{\rm exp}\left(-\frac{(\y-\x_i)^{T} \mathbf{A}(\x_i)^{-1} (\y-\x_i)}{2\delta^2}\right)\,d\y\\
&=\frac{2}{\delta^{2}\pi^{d/2}}\int_{\mathbb{R}^d} (u(\x_i)-u(\x_i+\sqrt2\delta \mathbf{A}^{1/2}\y)\,\text{\rm exp}\left(-\y^T\y\right)\,d\y.
\end{aligned}
\end{equation*}
Based on the analysis results of \cite{TrWe2014,EgLu1989},  a remarkable conclusion is that for any fixed horizon parameter $\delta>0$, such approximation is exponentially convergent with respect to the grid size if $u$ is  {\em analytic} in $\mathbb{R}^d$. Therefore, we expect the numerical solution produced by  the  linear collocation scheme \eqref{detailed_Interpolate_linear} to converge exponentially to the solution of the nonlocal diffusion model \eqref{dbc}  at the set of all grid points $\{\x_i\}$  since $|\mathcal{L}_{\delta}u(\x_i)- \mathcal{L}_{\delta,\alpha}u(\x_i)| $  and $|\mathcal{L}^h_{\delta}u(\x_i)- \mathcal{L}^h_{\delta,\alpha}u(\x_i)| $ are negligible. On the other hand, if the rectangular grid is not uniform, then the convergence order  will downgrade to the regular second-order for linear schemes since the exponential convergence doesn't exist anymore. These results will be numerically demonstrated  by experiments in  Section 6.1.

Ensuring asymptotic compatibility (i.e., the approximate solution of the nonlocal model problem \eqref{dbc} converges to the exact solution of the corresponding local PDE problem \eqref{acold} when the horizon parameter $\delta\rightarrow 0$ and the grid size $h\rightarrow 0$ simultaneously in arbitrary fashion)  is of fundamental importance in guaranteeing the accuracy and reliability of numerical schemes for nonlocal models.  Let us consider the simple one-dimensional problem ($d=1$) where $\mathbf{A}=1$ and a uniform partition with grid size $h$ is used. It holds 
by \cite{EgLu1989} and  the equation \eqref{modelerr} that if $u$ is analytic, then
\begin{equation}
\begin{aligned}
|\mathcal{L}^h_{\delta} u(x_i)-\mathcal{L}u(x_i)|&\leq |\mathcal{L}^h_{\delta} u(x_i)-\mathcal{L}_{\delta} u(x_i)|+|\mathcal{L}_{\delta} u(x_i)-\mathcal{L} u(x_i)|\\
&= O\Big(\Big(\frac{e^{-\pi/h}}{\delta}\Big)^2|u(x_i)-u(x_i+{\rm i}\sqrt2\delta\pi h)|\Big)+O(\delta^2)\rightarrow 0,
\end{aligned}
\end{equation}
as $\delta\rightarrow 0$ and $ e^{-\pi/h}/\delta\rightarrow 0$.  Thus we would like to say that our scheme \eqref{detailed_Interpolate_linear} is {\em effectively} asymptotically compatible since it additionally requires $ e^{-\pi/h}/\delta\rightarrow 0$.  Note  this condition can be easily satisfied in practice. 
In particular, when the ratio between $\delta$ and $h$  keeps fixed (i.e., $\delta/h=\kappa$ being a constant) along the grid refinement,
it is called the $\delta$-convergence  test.   This behavior is a manifestation of the well-known continuum limit in which the underlying physical system becomes increasingly smooth and continuous. For such case we have
$$|\mathcal{L}^h_{\delta} u(x_i)-\mathcal{L}u(x_i)| = O\Big(\Big(\frac{e^{-\pi/h}}{\kappa h}\Big)^2|u(x_i)-u(x_i+{\rm i}\sqrt2\kappa\pi )|\Big)+O(\kappa h^2) = O(h^2),$$
thus the numerical solution obtained for the  nonlocal diffusion model \eqref{dbc} will converge quadratically  to the solution of the corresponding local problem \eqref{dbc-pde} as $h\rightarrow 0$. It is expected such effective asymptotical compatibility will also show up in general dimensions and we will demonstrate it through extensive numerical  experiments in Section 6.2. 

\begin{rem}\label{quconv}
		For traditional bond-based nonlocal diffusion model, in order to achieve  asymptotic compatibility, a popular quadrature-based finite difference discretization for multidimensional problems  was developed in \cite{DuTao2019,DuJuTi2017}: for $i=1,\cdots,N_s$,
		\begin{equation} \label{acold}
			-\widetilde{\mathcal{L}}^h_{\delta,\alpha} u(\x_i):=\int_{B_{\delta,{\bf A},\alpha}(\x_i)} \mathcal{I}_h\left(\frac{u(\x_i)-u(\y)}{W(\x_i,\y)}\right)W(\x_i,\y)\gamma_{\alpha}(\x_i,\y)\,d\y= f(\x_i), 
		\end{equation}
		where the weight function $W(\x,\y)$ is usually given by
		\begin{equation*}
			W(\x,\y)=\frac{\left\|\y-\x\right\|^2}{\left\|\y-\x\right\|_1}
		\end{equation*}
		with the notation $\left\| \cdot \right\|_1$ standing for the $L^1$ norm. However, the numerical scheme  \eqref{acold} could not be asymptotically compatible for our 
		bond-based nonlocal diffusion model \eqref{dbc} and the original model \eqref{dbc-org} developed  in this paper. One of the main reasons is that the proof of asymptotic compatibility of  the scheme \eqref{acold} heavily depends
		on the radial symmetry (for isotropic diffusion) of the influence  region of the kernel function $\gamma_{\alpha}$, which doesn't hold anymore for the case of anisotropic diffusion for the proposed nonlocal model. Numerical tests presented in Section 6.2 will demonstrate this issue.
How to appropriately modify the scheme \eqref{acold} to make it to be 
asymptotically compatible remains an interesting open problem.
\end{rem}

\section{Numerical experiments}

In this section, we will conduct various numerical experiments in two and three dimensional space to showcase application of the proposed bond-based nonlocal diffusion model \eqref{dbc}, as  an approximation to the original nonlocal diffusion model  \eqref{dbc-org}, to both isotropic and anisotropic diffusion problems, and numerically demonstrate the convergence and effective asymptotic compatibility of the proposed linear collocation scheme \eqref{detailed_Interpolate_linear}. In addition, we also test the discrete maximum principle and the effect of the choice of the truncation parameter 
$\chi^2_\alpha(d)$ on  the approximation accuracy of $\mathcal{L}_{\delta,\alpha}$ to $\mathcal{L}_{\delta}$. Note that in Examples 1-5 we always take  
$\chi^2_\alpha(d)=36$ by default based on the comparison result observed from Example 6 in Subsection \ref{chicomp}. 
The choices of $\mathbf{A}$ in all examples satisfy the condition \eqref{condp}.
{The solution errors  are measured under the {\em discrete maximum norm}, which is defined to be the maximum absolute  value among all grid points $\{\x_i\}$.

\subsection{Tests with fixed horizon parameter}

We first keep the $\delta$ fixed and test the convergence of the linear collocation scheme \eqref{dbc} for solving the nonlocal diffusion model \eqref{dbc}.

\begin{exmp}\label{exp2}
Let us take the 2D domain $\Omega_s = [0, 1] \times [0, 1]$, $\mathbf{A}=[1,0;0,1]$,  and {$\delta=1/40$} for the nonlocal diffusion problem \eqref{dbc}. 
 The exact solution is chosen  to be $u(x_1,x_2)=x_1x_2^5$, $u(x_1,x_2)=e^{x_1 x_2}$, and $u(x_1,x_2)=\sin(x_1^2+x_2^2)$, respectively. Then
 the boundary value $g(x_1, x_2)$ are directly obtained from the exact solution $u(x_1,x_2)$ and the source term  $f(x_1,x_2)$ are determined accordingly based on the  original nonlocal diffusion model \eqref{dbc-org}.
\end{exmp}
	In the following, we perform numerical experiments with three types of rectangular grids.
 Firstly, we uniformly partition the domain to $N\times N$ rectanglur cells , where $N = 20, 25, 30, 35, 40, 50,$ respectively.  The grid spacing are clearly  the same for both $x$- and $y$- directions, 
 $h_x=h_y=h=1/N$. Table \ref{tab2-1} reports the discrete {$L^{\infty}$} solution errors  produced by the  linear collocation scheme \eqref{detailed_Interpolate_linear}.  We clearly observe the exponential  convergence  in all three exact solution cases as discussed in Section 5. Furthermore, it is also seen that after $N=40$ is reached, the errors don't change much anymore or even slightly increase. This is caused by the truncation of the influence region with the parameter $\chi^2_\alpha(2)=36$ to the original nonlocal model \eqref{dbc-org}. The model errors gradually dominate the solution errors along with the mesh refinement.

	\begin{table} [!ht]\small
		\center
		\renewcommand{\arraystretch}{1.2}
		\begin{tabular}{|c||c|c|c|c|c|c|}
		\hline
		& \multicolumn{2}{|c|}{$u(x_1,x_2)=x_1 x_2^5$} &  \multicolumn{2}{|c|}{$u(x_1,x_2)=e^{x_1 x_2}$} &  \multicolumn{2}{|c|}{$u(x_1,x_2)=\sin(x_1^2+x_2^2)$}  \\
		\hline
		$N$ &$L^\infty$ error&CR&$L^\infty$ error & CR &$L^\infty$ error& CR \\ \hline
		\hline 20 &  ${3.1594\times10^{-2}}$ & -& ${1.0367\times10^{-2}}$ &-&  ${2.0841\times10^{-2}}$ &-\\
		\hline 25 &  ${2.4209\times10^{-3}}$ & 11.51& ${8.6355\times10^{-4}}$ &11.14&  ${1.9893\times10^{-3}}$ &10.53\\
		\hline 30& ${1.2152\times10^{-4}}$  &16.41&${4.3467\times10^{-5}}$ &16.39& ${1.0027\times10^{-4}}$ & 16.39\\
		\hline 35& ${ 3.0658\times10^{-6}}$  & 23.87&${9.5991\times10^{-7}}$ &24.74& ${2.4812\times10^{-6}}$ & 24.00\\
		\hline 40 & ${7.5238\times 10^{-8}}$ &27.76 & ${9.6315\times 10^{-8}}$ & 17.22& ${2.8200\times 10^{-7}}$  &16.29 \\
		\hline 50 & ${9.5161\times 10^{-8}}$ &-1.05& ${5.8425\times 10^{-8}}$ &2.24 & ${3.1212\times 10^{-7}}$  &-0.46 \\			
		\hline
	\end{tabular}
		\caption{Numerical results on the discrete $L^{\infty}$  solution errors and corresponding convergence rates  produced by the linear collocation scheme \eqref{detailed_Interpolate_linear} for the nonlocal diffusion model \eqref{dbc} with fixed {$\delta=1/40$} in Example \ref{exp2}. Uniform rectangular grids of $N\times N$ are used. }
		\label{tab2-1}
	\end{table}

	Secondly, we discretize the domain uniformly  into a grid with $N_x$ intervals in the $x$-direction and $N_y$ intervals in the $y$-direction, where $(N_x,Ny)= (40,20), (50,25), (60,30), (70,35), (80,40)$, respectively. Note that the grid spacings are now different  for $x$- and $y$-directions with $h_x=1/N_x$ and $h_y=1/N_y$. The discrete  $L^{\infty}$ solution errors, as reported in Table \ref{tab2-2}, still  show exponential convergence as expected. 
	
	\begin{table} [!ht]\small
	\center
	\renewcommand{\arraystretch}{1.2}
	\begin{tabular}{|c|c||c|r|c|r|c|r|}
		\hline
		&& \multicolumn{2}{|c|}{$u(x_1,x_2)=x_1 x_2^5$} &  \multicolumn{2}{|c|}{$u(x_1,x_2)=e^{x_1 x_2}$} &  \multicolumn{2}{|c|}{$u(x_1,x_2)=\sin(x_1^2+x_2^2)$}  \\
		\hline
		$N_x$ &$N_y$&$L^\infty$ error &CR &$L^\infty$ error &CR &$L^\infty$ error&CR \\ \hline
		\hline 40 &20& ${3.3687\times 10^{-2}}$ &-& ${4.7600\times 10^{-3}}$&- & ${9.8815\times 10^{-3}}$ &-  \\
		\hline 50 &25& $ {2.6168\times 10^{-3}}$&11.45 & ${4.6607\times 10^{-4}}$&10.41 & ${1.0111\times 10^{-3}}$ &10.21 \\
		\hline 60 & 30&$ {1.2902\times 10^{-4}}$ &16.51& ${2.3552\times 10^{-5}}$ &16.37& ${5.1047\times 10^{-5}}$&16.38  \\
		 \hline 70 & 35&$ {3.2314\times 10^{-6}}$&23.92& ${6.3230\times 10^{-7}}$ & 23.47& ${1.1627\times 10^{-6}}$ &24.53 \\
	    \hline 80 & 40&$ {1.7862\times 10^{-7}}$ & 21.68& ${6.7596\times 10^{-8}}$ &16.74& ${3.3266\times 10^{-7}}$&9.37  \\
		\hline
	\end{tabular}
	\caption{Numerical results on the  discrete  $L^{\infty}$  solution errors and corresponding convergence rates  produced by the linear collocation scheme \eqref{detailed_Interpolate_linear} for the nonlocal diffusion model \eqref{dbc} with fixed {$\delta=1/40$} in Example \ref{exp2}. Uniform rectangular grids of $N_x\times N_y$ are used. }
	\label{tab2-2}
\end{table}

	Finally, we divide the domain in the $x$-direction non-uniformly: $[0,0.5]$ is partitioned  to $N_{xl}$ subintervals and $[0.5,1)$ $N_{xr}$ subintervals. The $y$-direction of the domain is then partitioned uniformly into $N_{y}$ subintervals. Consequently, the obtained rectangular grids  are globally  non-uniform with $(N_{xl}+N_{xr})\times N_y$ rectangular cells. Table \ref{tab2-3} presents the  discrete  $L^{\infty}$ solution errors obtained by using the linear collocation scheme \eqref{detailed_Interpolate_linear}.  We now observe only second-order convergence  in all cases as discussed in Section 5.

	\begin{table} [!ht]\small
	\center
	\renewcommand{\arraystretch}{1.2}
	\begin{tabular}{|c|c|c||c|c|c|c|c|c|}
		\hline
		&&& \multicolumn{2}{|c|}{$u(x_1,x_2)=x_1 x_2^5$} &  \multicolumn{2}{|c|}{$u(x_1,x_2)=e^{x_1 x_2}$} &  \multicolumn{2}{|c|}{$u(x_1,x_2)=\sin(x_1^2+x_2^2)$}  \\
		\hline
		$N_{xl}$ &$N_{xr}$&$N_{y}$&$L^\infty$ error&CR&$L^\infty$ error &CR&$L^\infty$ error&CR \\ \hline
		\hline 10&15 &20&  ${3.0391\times10^{-2}}$ & - &${2.2799\times10^{-2}}$ & -&${1.4295\times10^{-2}}$ &-\\
		\hline 20&30&40& ${8.2144\times10^{-4}}$  & 5.21&${1.5704\times10^{-3}}$ &3.86&${1.1348\times10^{-3}}$&3.66 \\
		\hline 40&60&80& ${1.8793\times 10^{-4}}$  &2.13& ${3.7312\times 10^{-4}}$ &2.07&${2.7345\times 10^{-4}}$  &2.05  \\
		\hline 80&120 &160& $ {4.5641\times 10^{-5}}$ &2.04&${9.1977\times 10^{-5}}$ & 2.02&${6.7871\times 10^{-5}}$&2.01 \\
		\hline 160&240& 320&$ {1.1393\times 10^{-5}}$ & 2.00&${2.2804\times 10^{-5}}$ & 2.01&${1.6774\times 10^{-5}}$ &2.01 \\
		\hline
	\end{tabular}
	\caption{Numerical results on the  discrete  $L^{\infty}$  solution errors and corresponding convergence rates with fixed {$\delta=1/40$} produced by the linear collocation scheme \eqref{detailed_Interpolate_linear} for the nonlocal diffusion model \eqref{dbc} in Example \ref{exp2}. Non-uniform rectangular grids of $(N_{xl}+N_{xr})\times N_y$ are used.}
	\label{tab2-3}
\end{table}
\subsection{Tests for $\delta$-convergence}

We now investigate the effective asymptotic compatibility of the proposed  collation scheme \eqref{detailed_Interpolate_linear}  as $\delta$ and $h$ approach zero simultaneously. Specifically, we consider the so-called $\delta$-convergence  test  and set the ratio $\delta/h=1$.
We compute the error between the  numerical solution obtained for the nonlocal problem \eqref{dbc} and the exact solution of the corresponding classical local problem \eqref{dbc-pde}
and observe its behaviors as  $h$ tend to zero. 
By characterizing the $\delta$-convergence of the proposed collocation scheme, we can assess its accuracy and reliability, and gain insight into the underlying physics of the system under consideration.

\begin{exmp}\label{exp4}
 We consider the 2D domain $\Omega_s = [0, 1] \times [0, 1]$ and  take $u(x_1,x_2)=\sin(x_1^2+x_2^2)$ as the exact solution of the classic diffusion (PDE) problem \eqref{dbc-pde}.
 The following three different $2\times 2$ constant coefficient matrices are tested:
 \begin{equation*}
\mathbf{A}_1=	
\left(\begin{array}{cc}
1& 0\\[3pt]
0 & 1
\end{array}\right),\qquad\mathbf{A}_2=     \left(\begin{array}{cc}
10& 0 \\[3pt]
0 & 1
\end{array}\right),
\end{equation*}
and
\begin{equation*}
\mathbf{A}_3=	\left(\begin{array}{cc}
\cos \frac{\pi}{6}& \sin \frac{\pi}{6} \\[3pt]
-\sin \frac{\pi}{6} & \cos\frac{\pi}{6}
\end{array}\right)\left(\begin{array}{cc}
10 & 0 \\[3pt]
0 & 1
\end{array}\right)\left(\begin{array}{cc}
\cos \frac{\pi}{6} & -\sin\frac{\pi}{6} \\[3pt]
\sin \frac{\pi}{6} & \cos \frac{\pi}{6}
\end{array}\right).
\end{equation*}
 For the nonlocal diffusion model \eqref{dbc}, 
 the boundary value $g(x_1, x_2)$ are directly obtained from the exact solution $u(x_1,x_2)$ and the source term  $f(x_1,x_2)$ is determined accordingly based on the classic diffusion problem \eqref{dbc-pde}.
\end{exmp}

For $\mathbf{A}_1$, we have the influence region of the kernel function $\gamma_{\alpha}(\x,\y)$ associated with a point $\x\in\Omega_s$ as:  $$B_{\delta,{\bf A}_1,\alpha}(\x)=\left\{(y_1,y_2)\,\Big|\,\frac{(y_{1}-x_{1})^2}{36}+ \frac{(y_{2}-x_{2})^2}{36}\le \delta^2\right\},$$ which corresponds to the disk of radius $6\delta$ centered at $\x$. We then divide the domain $\Omega_s$ into  a uniform grid of $N\times N$ cells with the grid size $h=1/N$, where $N=40, 80, 160, 320$ respectively.  We  set $\delta=h$ to maintain a consistent ratio across all levels of  grids. Figure \ref{meat}-(a) illustrates the computational domain $\Omega=\Omega_s\cup\Omega_c$ and the influence regions for ${\bf A}_1$. The influence region for $\mathbf{A}_2$  is given by $$B_{\delta,{\bf A}_2,\alpha}(\x)=\left\{(y_1,y_2)\,\Big|\,\frac{(y_{1}-x_{1})^2}{360}+ \frac{(y_{2}-x_{2})^2}{36}\le \delta^2\right\},$$ which is the elliptic region centered at $\x$ with the major axis of length $12\sqrt{10}\delta$ and the minor axis of length $12\delta$. Figure \ref{meat}-(b)  illustrates the computational domain $\Omega$ and the influence regions for ${\bf A}_2$. The influence area for $\mathbf{A}_3$ is given by 
\begin{equation}
B_{\delta,{\bf A}_3,\alpha}(\x)=\left\{(y_1,y_2)\,\Big|\,\frac{13(y_{1}-x_{1})^2}{1440}+\frac{\sqrt{3}(y_{1}-x_{1})(y_{2}-x_{2})}{80} +\frac{31(y_{2}-x_{2})^2}{1440}\le \delta^2\right\}.
\end{equation}
The diagonalization of $\mathbf{A}_3$ shows that $B_{\delta,{\bf A}_3,\alpha}(\x)$ is obtained by rotating the elliptic region corresponding to $\mathbf{A}_2$ with an angle of $30^\circ$ counterclockwisely.
Figure \ref{meat}-(c)   shows  the computational domain $\Omega$ and the influence regions for ${\bf A}_3$.

Table \ref{monday} reports the  discrete {$L^{\infty}$} solution errors and corresponding convergence rates produced by the   linear collocation scheme \eqref{linearsys}. We observe the  second-order convergence along the refinement for all these  test cases of $\delta$-convergence as discussed in Remark \ref{quconv}.

\begin{figure}[!ht]
\centering
\subfigure[\label{fig:a1} The coefficient matrix ${\bf A}_1$]{
\includegraphics[scale=0.25]{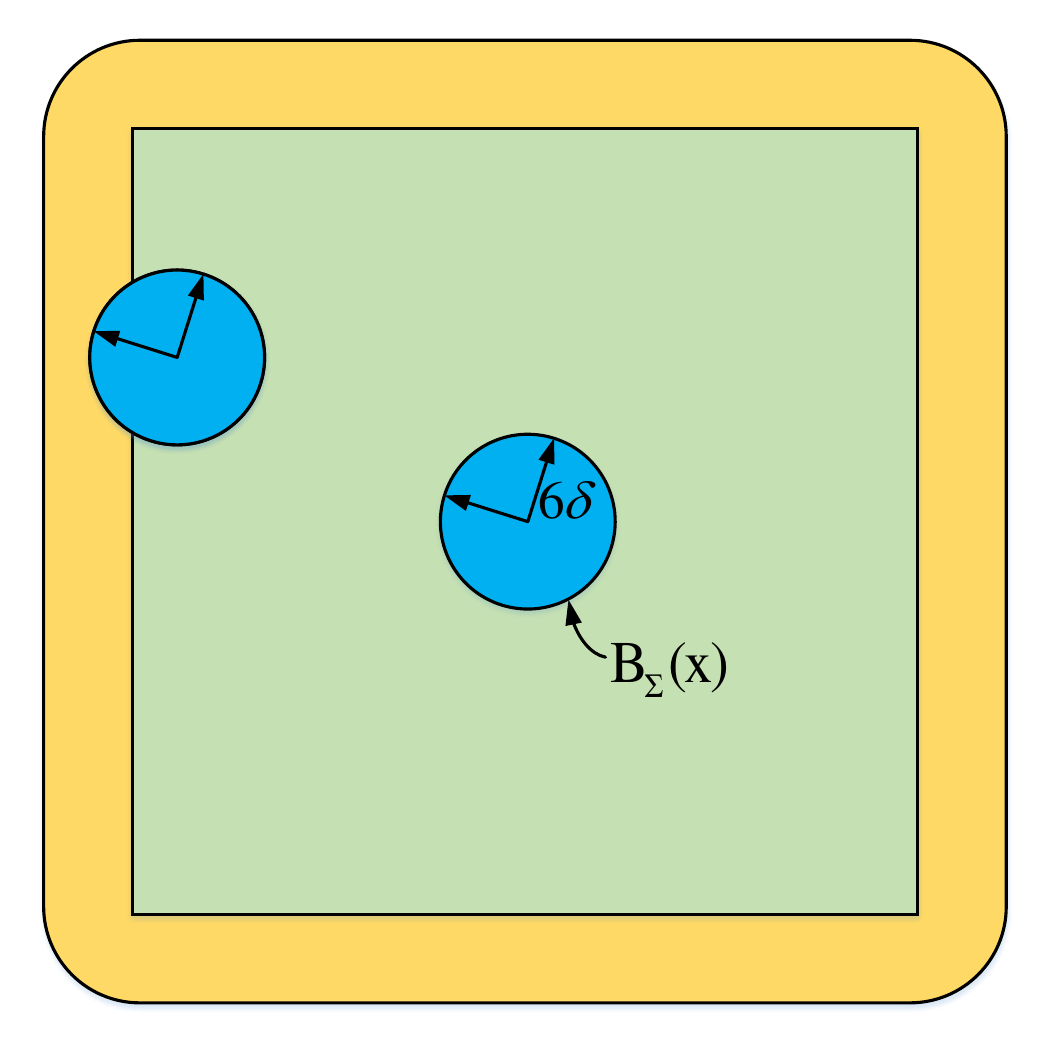}\includegraphics[scale=0.145]{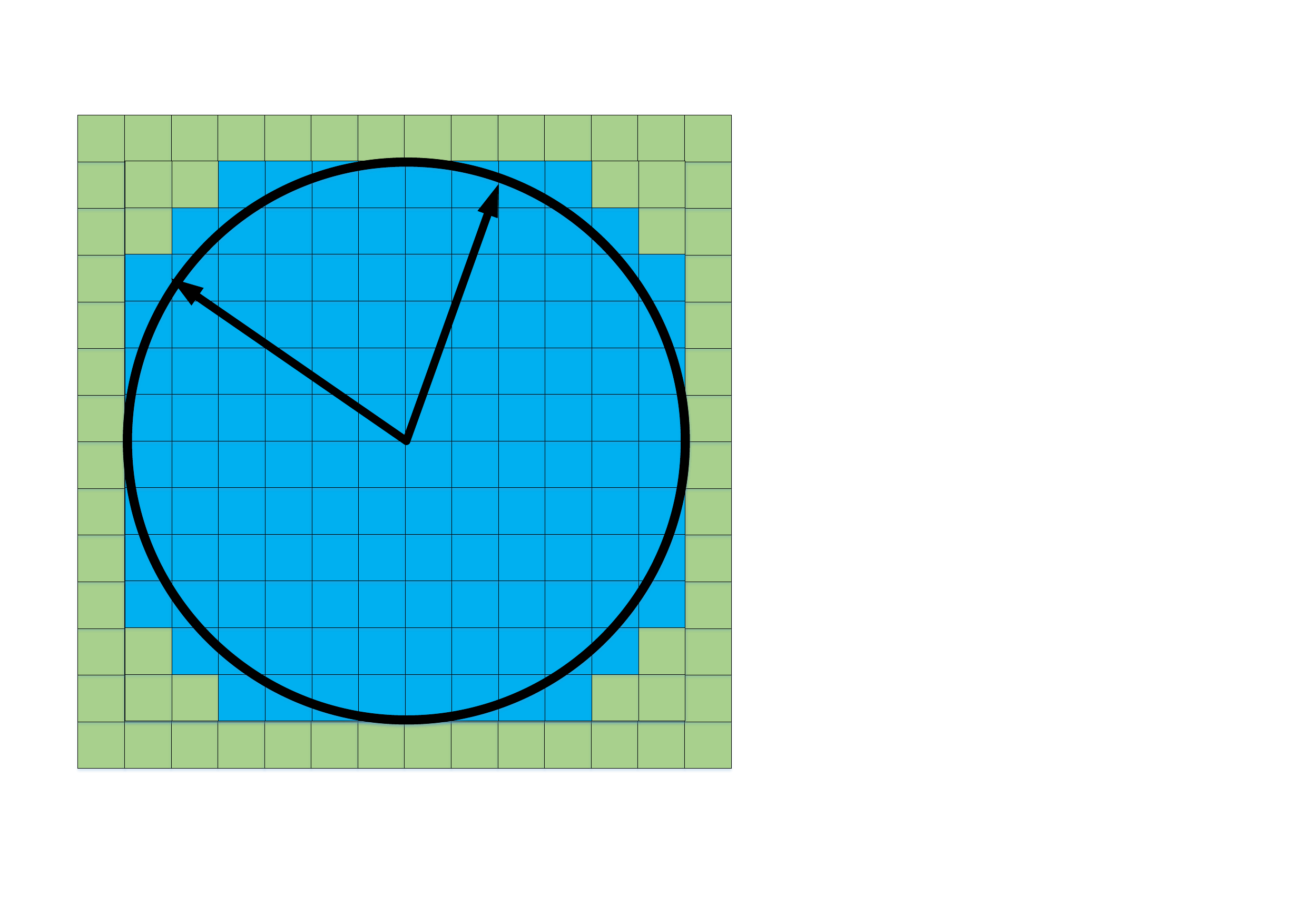}}
\subfigure[\label{fig:a2} The coefficient matrix  ${\bf A}_2$]{
\includegraphics[scale=0.25]{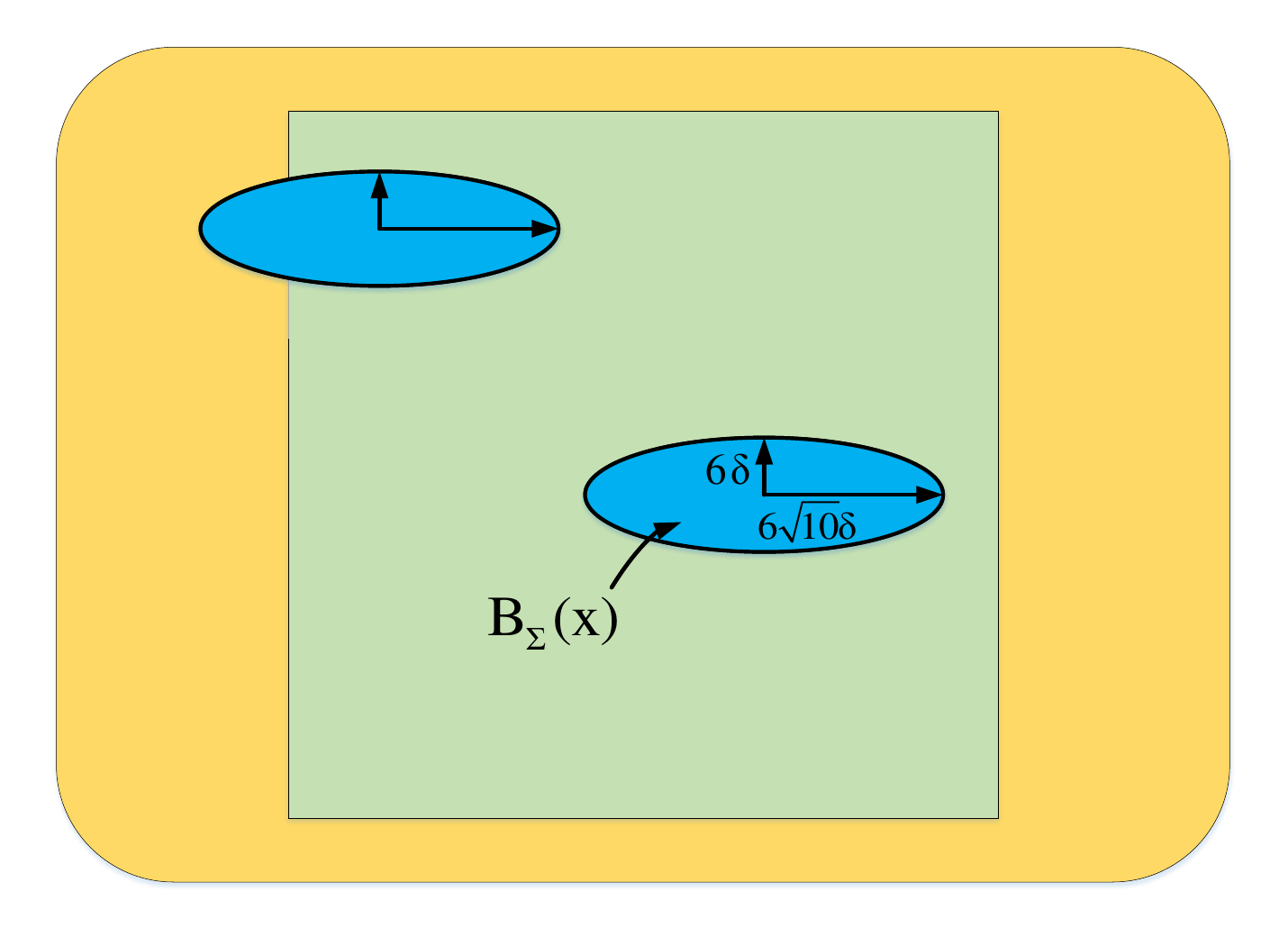}\includegraphics[scale=0.145]{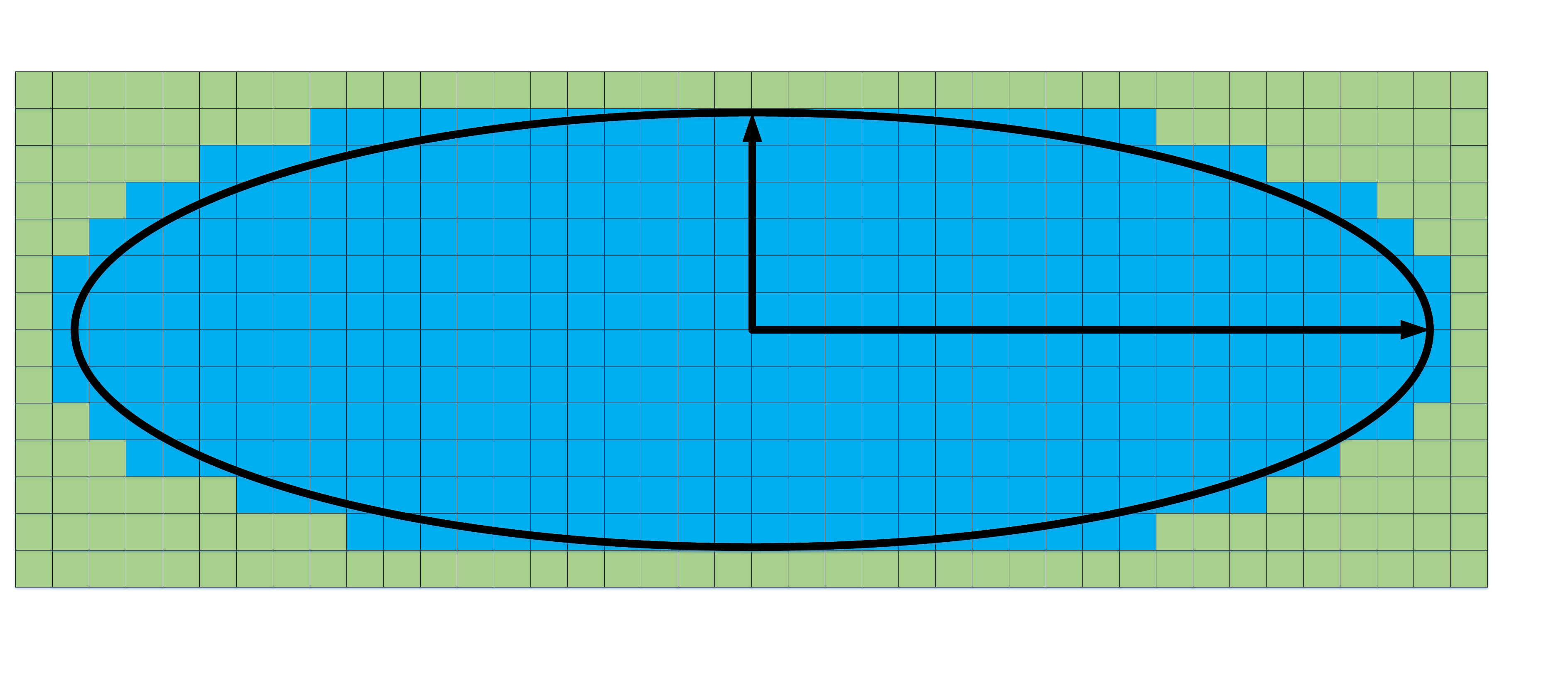}}
\subfigure[\label{fig:a3}The coefficient matrix  ${\bf A}_3$]{
\includegraphics[scale=0.25]{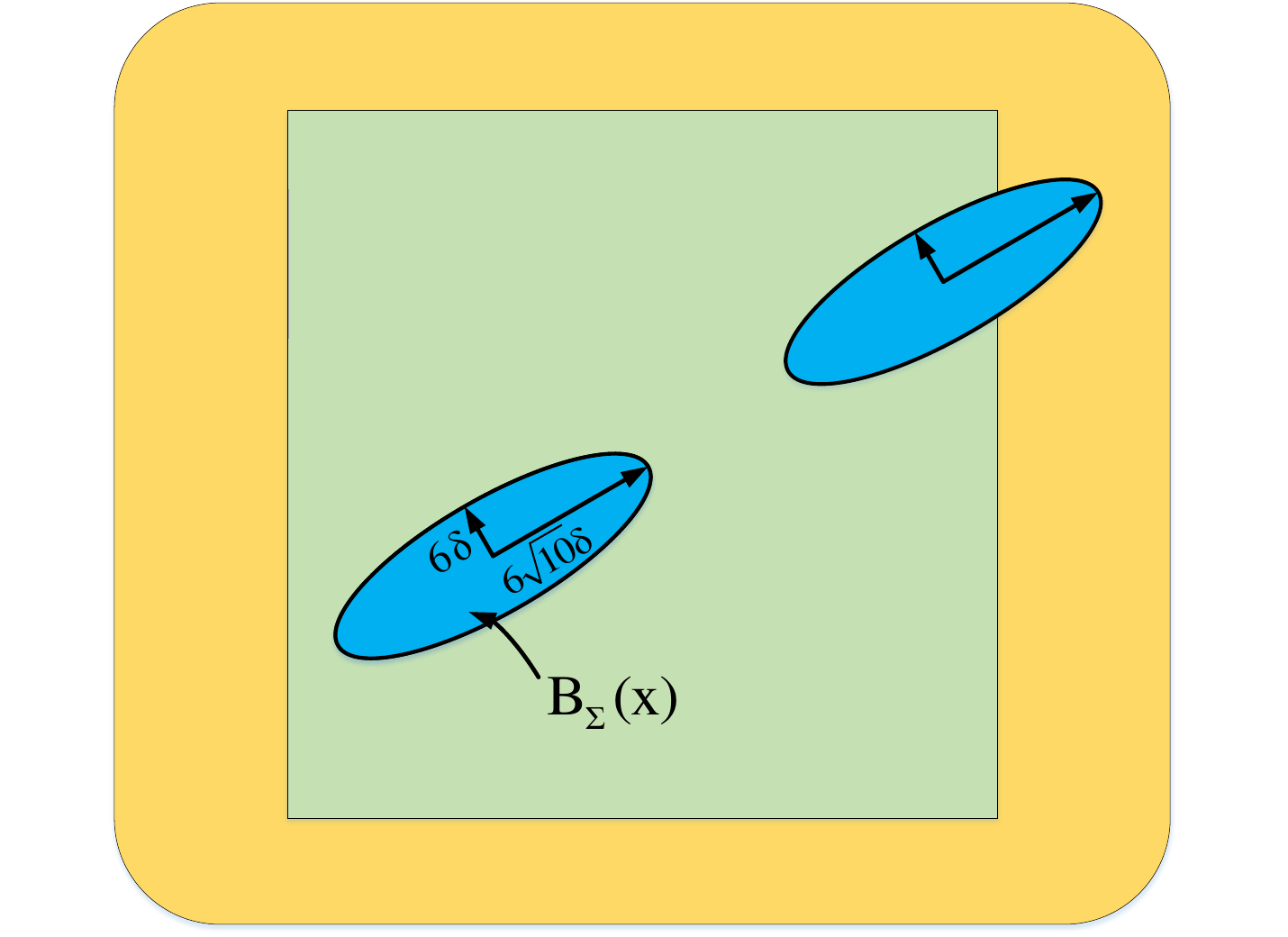}\includegraphics[scale=0.1]{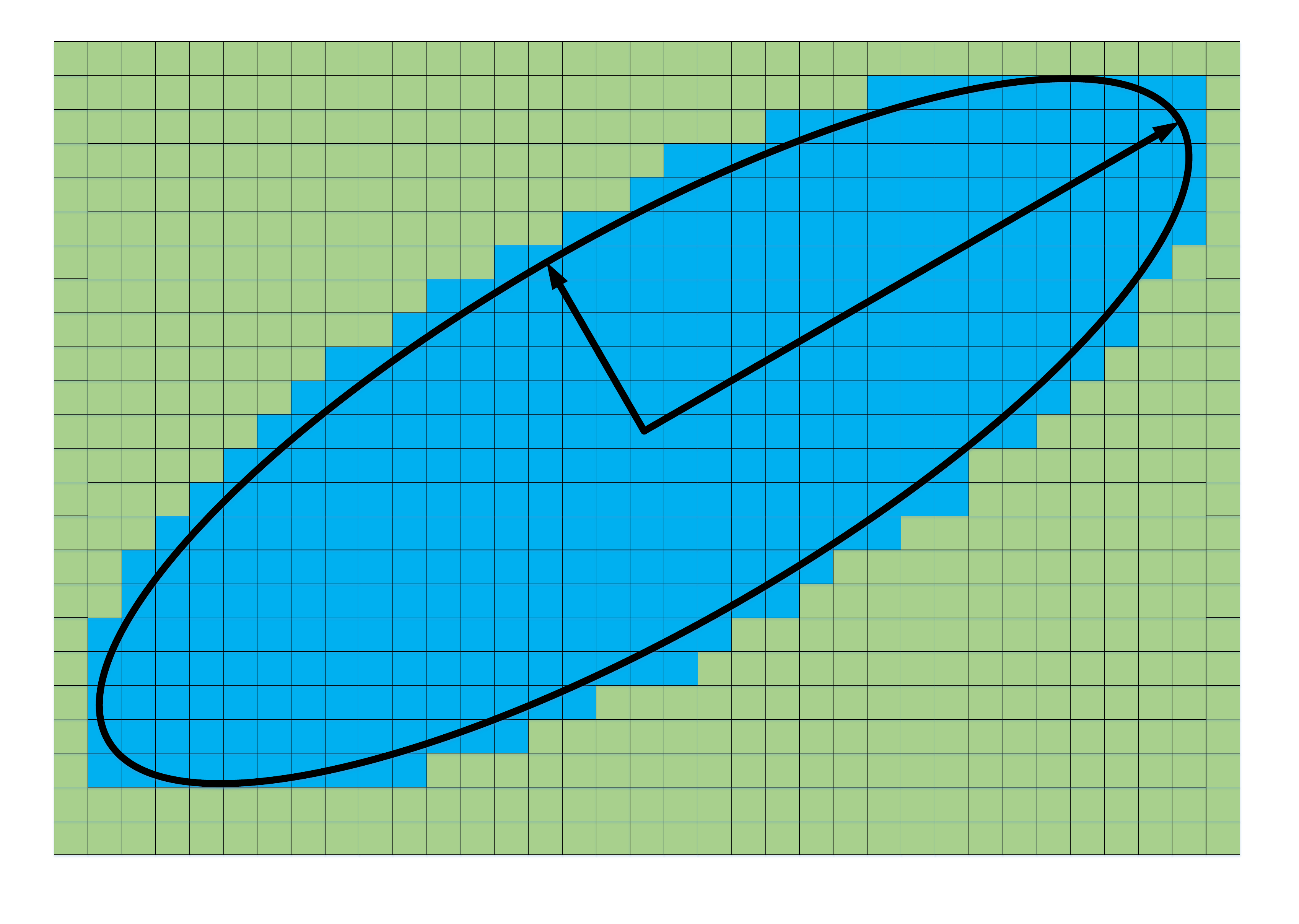}}
\caption{Illustration of the computational domain $\Omega=\Omega_s\cup\Omega_c$ and the influence regions with grids in the background.}
\label{meat}
\end{figure}

\begin{table}[!ht]\small
\center
\renewcommand{\arraystretch}{1.2}
\begin{tabular}{|c||c|c|c|c|c|c|}
\hline $\mathrm{N}$ & $\mathbf{A}_1$ &CR& $\mathbf{A}_2$ &CR& $\mathbf{A}_3$&CR \\\hline 
\hline 40 & $4.8765\times10^{-4}$ &-& $3.5896\times10^{-3}$&- & $2.8574\times10^{-3}$&- \\
\hline 80 & $1.1504\times10^{-4}$ & {2.08}&$8.4238\times10^{-4}$ &{2.09}& $6.7127\times10^{-4}$ &{2.09}\\
\hline 160 & $2.8346\times10^{-5}$& {2.02} & $2.0492\times10^{-4}$ &{2.04}& $1.5967\times10^{-4}$ & {2.07}\\
\hline 320 & $6.6997\times10^{-6}$ &{2.08}& $4.5901\times10^{-5}$ & {2.16}& $3.9727\times10^{-5}$ & {2.01}\\
\hline
\end{tabular}
\caption{Numerical results on the  discrete  {$L^{\infty}$}   solution errors and corresponding convergence rates produced by the linear collocation scheme \eqref{detailed_Interpolate_linear} for the nonlocal diffusion model \eqref{dbc} with $\delta=h$  in Example \ref{exp4}.}
\label{monday}
\end{table}

{For comparison purpose, the numerical results obtained using the quadrature-based finite difference discretization \eqref{acold} are also presented in Table \ref{zhouwu}. It is evident from the results that the scheme  \eqref{acold} produces second-order convergence for the isotropic diffusion case (${\mathbf A}_1$), but  fails to converge for the anisotropic diffusion cases (${\mathbf A}_2$ and ${\mathbf A}_3$)} as discussed in Remark \ref{quconv}.
\begin{table}[h]\small
	\center

	\renewcommand{\arraystretch}{1.2}	
		\begin{tabular}{|c||c|c|c|c|c|c|}
		\hline $\mathrm{N}$ & $\mathbf{A}_1$ &CR& $\mathbf{A}_2$ &CR& $\mathbf{A}_3$&CR \\\hline 
		\hline 40 & $5.7139\times10^{-4}$ &-& $4.1611\times10^{-3}$&- & $4.4236\times10^{-3}$&- \\
		\hline 80 & $1.3425\times10^{-4}$ & 2.09&$1.8107\times10^{-3}$ &1.20& $2.3123\times10^{-3}$ &0.94\\
		\hline 160 & $3.3151\times10^{-5}$& 2.02 & $1.3431\times10^{-3}$ & 0.43& $1.8423\times10^{-3}$ &0.33\\
		\hline 320 & $7.8984\times10^{-6}$ & 2.07& $1.2238\times10^{-3}$ &0.13 & $1.7201\times10^{-3}$ &0.10\\
		\hline
	\end{tabular}
		\caption{Numerical results on the discrete  $L^{\infty}$  solution errors and corresponding convergence rates  produced by the quadrature-based finite difference discretization scheme (\ref{acold}) for the nonlocal diffusion model \eqref{dbc} with  $\delta=h$ in Example \ref{exp4}.}
		\label{zhouwu}
	
\end{table}

\begin{exmp}\label{exp3D}
We consider the 3D domain $\Omega_s = [0, 1] \times [0, 1]\times [0, 1]$ and take $u(x_1,x_2,x_3)=\sin(x_1^2+x_2^2+x_3^2)$ as the exact solution of the classic diffusion (PDE) problem \eqref{dbc-pde}.
 The following three different $3\times 3$ constant coefficient matrices are tested:
 \begin{equation*}
\mathbf{A}_1=	
\left(\begin{array}{ccc}
1& 0&0\\[3pt]
0 & 1&0\\[3pt]
0&0&1
\end{array}\right),\qquad\mathbf{A}_2=     \left(\begin{array}{ccc}
4& 0&0\\[3pt]
0 & 1&0\\[3pt]
0&0&1\end{array}\right),\quad 
\end{equation*}
and
 \begin{equation*}
{\mathbf{A}_3=\left(\begin{array}{ccc}
\cos\frac{\pi}{4} &-\sin\frac{\pi}{4}& 0\\[3pt]
\sin\frac{\pi}{4}&\cos\frac{\pi}{4} &0\\[3pt]
0&0&1
\end{array}\right).
\left(\begin{array}{ccc}
	4 &0& 0\\[3pt]
	0&1&0\\[3pt]
	0&0&1
\end{array}\right)
\left(\begin{array}{ccc}
\cos\frac{\pi}{4} &\sin\frac{\pi}{4}& 0\\[3pt]
-\sin\frac{\pi}{4}&\cos\frac{\pi}{4} &0\\[3pt]
0&0&1
\end{array}\right).}
\end{equation*}

 For the nonlocal diffusion model \eqref{dbc},  the boundary value $g(x_1, x_2,x_3)$ are directly obtained from the exact solution $u(x_1,x_2,x_3)$ and the source term  $f(x_1,x_2,x_3)$ are determined accordingly based on the classic diffusion problem \eqref{dbc-pde}.
\end{exmp}

In this example,  the influence region for $\mathbf{A}_1$ associated with a point $\x\in\Omega_s$ is given by 
 $$B_{\delta,{\bf A}_1,\alpha}(\x)=\left\{(y_1,y_2,y_3)\,\Big|\,\frac{(y_{1}-x_{1})^2}{36}+ \frac{(y_{2}-x_{2})^2}{36}+\frac{(y_{3}-x_{3})^2}{36}\le \delta^2\right\},$$ which corresponds to the ball of radius $6\delta$ centered at $\x$. The influence region for $\mathbf{A}_2$ is given by $$B_{\delta,{\bf A}_2,\alpha}(\x)=\left\{(y_1,y_2,y_3)\,\Big|\,\frac{(y_{1}-x_{1})^2}{144}+ \frac{(y_{2}-x_{2})^2}{36}+\frac{(y_{3}-x_{3})^2}{36}\le \delta^2\right\},$$ which is the ellipsoid region centered at $\x$ with the three axises of length $24\delta$, $12\delta$, and $12\delta$, respectively. The influence area for $\mathbf{A}_3$ is given by 
\begin{equation}
B_{\delta,{\bf A}_3,\alpha}(\x)=\left\{(y_1,y_2,y_3)\,\Big|\,\frac{5(y_{1}-x_{1})^2}{288\delta^2}-\frac{(y_{1}-x_{1})(y_{2}-x_{2})}{48}+\frac{5(y_{2}-x_{2})^2}{288}+\frac{(y_{3}-x_{3})^2}{36}\le \delta^2\right\}.
\end{equation}
The diagonalization of $\mathbf{A}_3$ shows that $B_{\delta,{\bf A}_3,\alpha}(\x)$ can be obtained by rotating the ellipsoid region corresponding to  $\mathbf{A}_2$ clockwisely with an angle of $45^\circ$ along the $z$-axis. 
 We divide the domain $\Omega_s$ into a uniform grid of $N\times N \times N$ cells with the grid size $h=1/N$, where $N=20,30,40,50$,  respectively. We still set $\delta=h$ on all level of grids. Table \ref{today} reports the  discrete {$L^{\infty}$}  solution errors and corresponding convergence rates produced by the  linear collocation scheme \eqref{linearsys}. We again observe the  second-order convergence along the refinement for these $\delta$-convergence tests.

\begin{table}[!ht]\small
\center
\renewcommand{\arraystretch}{1.2}
\begin{tabular}{|c||c|c|c|c|c|c|}
\hline $\mathrm{N}$ & $\mathbf{A}_1$ &CR& $\mathbf{A}_2$ &CR& $\mathbf{A}_3$&CR \\ \hline 
\hline 20 & $2.4291\times10^{-3}$&-& $6.1233\times10^{-3}$ &-& $6.9361\times10^{-3}$&- \\
\hline 30 & $1.0677\times10^{-3}$ &2.03& $2.7587\times10^{-3}$ &1.97& $3.041\times10^{-3}$&2.03 \\
\hline 40 & $6.1470\times10^{-4}$ &1.92& $1.4588\times10^{-3}$ &2.21& $1.6902\times10^{-3}$ &2.04\\
\hline 50 & $3.8099\times10^{-4}$ &2.14& $9.2431\times10^{-4}$ &2.05& $1.0767\times10^{-3}$ &2.02\\
\hline
\end{tabular}
\caption{Numerical results on the discrete  {$L^{\infty}$}   solution errors and corresponding convergence rates  produced by the linear collocation scheme \eqref{detailed_Interpolate_linear} for the nonlocal diffusion model \eqref{dbc} with  $\delta=h$ in Example \ref{exp3D}.}
\label{today}
\end{table}

\begin{exmp}\label{expvar}
In this example we consider the case of variable diffusion coefficient matrix.
We consider the 2D domain $\Omega_s = [0, 1] \times [0, 1]$ and take $u(x_1,x_2)=\sin(x_1^2+x_2^2)$ as the exact solution of the classic diffusion (PDE) problem \eqref{dbc-pde}.
The following two different  variable $2\times 2$  coefficient matrices are considered:
 \begin{equation*}
\mathbf{A}_1(x_1,x_2)=\left(\begin{array}{cc}
k_1(x_1,x_2) & 0 \\[3pt]
0 & k_2(x_1,x_2)
\end{array}\right)
\end{equation*}
and
\begin{equation*}
\mathbf{A}_2(x_1,x_2)=\left(\begin{array}{cc}
\cos \frac{5\pi}{12} & \sin \frac{5\pi}{12} \\[3pt]
-\sin \frac{5\pi}{12} & \cos \frac{5\pi}{12}
\end{array}\right)\left(\begin{array}{cc}
k_1(x_1,x_2) & 0 \\[3pt]
0 & k_2(x_1,x_2)
\end{array}\right)\left(\begin{array}{cc}
\cos \frac{5\pi}{12} & -\sin \frac{5\pi}{12} \\[3pt]
\sin \frac{5\pi}{12} & \cos \frac{5\pi}{12}
\end{array}\right),
\end{equation*}
where $k_1(x_1,x_2)=4-2 x_1^2-x_2^2$ and  $k_2(x,x_2)=4-x_1^2-2 x_2^2$. 
For the nonlocal diffusion model \eqref{dbc}, the boundary value $g(x_1, x_2)$ are directly obtained from the exact solution $u(x_1,x_2)$ and the source term  $f(x_1,x_2)$ are determined accordingly based on the classic diffusion problem \eqref{dbc-pde}.
\end{exmp}

The shape of the influence region associated with a point in $\Omega_s$ varies from place to place in this example. For example, for $\mathbf{A}_1(x_1,x_2)$, the  influence regions at the point $(1,1)$ and $(0,0)$ are a circle, but the  influence  at the point $(0,1)$ is an elliptic region. The influence region of $\mathbf{A}_2(x_1,x_2)$ can be obtained by counterclockwisely rotating the influence region corresponding to $\mathbf{A}_1(x,y)$ with an angle of  $75^\circ$. We use
a uniform grid of $N\times N$ cells with  the grid size $h=1/N$, where $N=40, 80, 160, 320$, respectively. We again set $\delta=h$ on all level of grids.
Table \ref{tusu} reports the  discrete {$L^{\infty}$}  solution errors and corresponding convergence rates produced by the   linear collocation scheme \eqref{detailed_Interpolate_linear}. We again observe the  second-order convergence along the refinement for these $\delta$-convergence test cases although the diffusion coefficient matrices are varying now.

\begin{table}[!ht]\small
\center

\renewcommand{\arraystretch}{1.2}
\begin{tabular}{|c||c|c|c|c|}
\hline $\mathrm{N}$ & $\mathbf{A}_1(x_1,x_2)$ &CR& $\mathbf{A}_2(x_1,x_2)$ &CR \\\hline 
\hline 40 & $1.4229\times10^{-3}$ &-& $1.4595\times10^{-3}$&-  \\
\hline 80& $3.5517\times10^{-4}$&2.00 & $3.6506\times10^{-4}$  &2.00\\
\hline 160 & $8.5297\times10^{-5}$ &2.06& $8.9876\times10^{-5}$& 2.02\\
\hline 320 & $2.0096\times10^{-5}$ &2.09& $2.5342\times10^{-5}$&1.83\\
\hline
\end{tabular}
\caption{Numerical results on the  discrete {$L^{\infty}$}   solution errors and corresponding convergence rates  produced by the linear  collocation scheme \eqref{detailed_Interpolate_linear} for the nonlocal diffusion model \eqref{dbc} with $\delta=h$ in Example \ref{expvar}. Note 
that the two diffusion coefficient matrices are varying.}
\label{tusu}

\end{table}

\subsection{Test of discrete maximum principle}

In this subsection we test the discrete maximum principle of the linear collocation scheme \eqref{linearsys}. 
\begin{exmp}\label{expDis}
We take the 2D domain $\Omega_s = [0, 1] \times [0, 1]$ and $\delta =1/40$ for the nonlocal diffusion problem \eqref{dbc}. Four different $2\times 2$ coefficient matrices (two constant and two variable matrices) are considered:
$$
\begin{aligned}
& \mathbf{A}_1=\left(\begin{array}{cc}
10 & 0 \\[3pt]
0 & 1
\end{array}\right),\\[3pt]
& \mathbf{A}_2=\left(\begin{array}{cc}
\cos \frac{\pi}{6} & \sin \frac{\pi}{6} \\[3pt]
-\sin \frac{\pi}{6} & \cos \frac{\pi}{6}
\end{array}\right)\left(\begin{array}{cc}
10 & 0 \\[3pt]
0 & 1
\end{array}\right)\left(\begin{array}{cc}
\cos \frac{\pi}{6} & -\sin \frac{\pi}{6} \\[4pt]
\sin \frac{\pi}{6} & \cos \frac{\pi}{6}
\end{array}\right),\\[3pt]
& \mathbf{A}_3(x_1,x_2)=\left(\begin{array}{cc}
k_1(x_1,x_2)& 0 \\[3pt]
0 & k_2(x_1,x_2)
\end{array}\right),\\[3pt]
& \mathbf{A}_4(x_1,x_2)=\left(\begin{array}{cc}
\cos \frac{5\pi}{12} & \sin \frac{5\pi}{12} \\[3pt]
-\sin \frac{5\pi}{12} & \cos \frac{5\pi}{12}
\end{array}\right)\left(\begin{array}{cc}
k_1(x_1,x_2) & 0 \\
0 & k_2(x_1,x_2)
\end{array}\right)\left(\begin{array}{cc}
\cos \frac{5\pi}{12} & -\sin \frac{5\pi}{12} \\[3pt]
\sin \frac{5\pi}{12} & \cos \frac{5\pi}{12}
\end{array}\right),
&
\end{aligned}
$$
where $k_1(x_1,x_2)=4-2 x_1^2-x_2^2$ and  $k_2(x_1,x_2)=4-x_1^2-2 x_2^2$.
In this example, we impose a Dirichlet boundary condition on $\Omega_c$ as follows:  $u(x_1, x_2) = 1$ if $x_1 \leq 0$ or $x_2 \leq 0$ and $u(x_1, x_2) = 0$ otherwise.  The source term  is chosen as $f(x_1,x_2)=0$. Although the exact solution of of the nonlocal diffusion problem \eqref{dbc} is unknown, we know its value must fall between 0 and 1. 
\end{exmp}

We take a  uniform grid of $80\times80$ and solve the nonlocal model using the linear collocation scheme \eqref{linearsys}. 
Figure \ref{talk} plots the numerical solutions produced with these four different coefficient matrices. We can see that  the discrete
 maximum principle is well preserved in all cases.

\begin{figure}[!ht] 
\centering{
\subfigure[\label{fig:aa1} The coefficient matrix ${\bf A}_1$]{
	\begin{minipage}[t]{.46\linewidth}
		\centering
		\includegraphics[width=.95\linewidth]{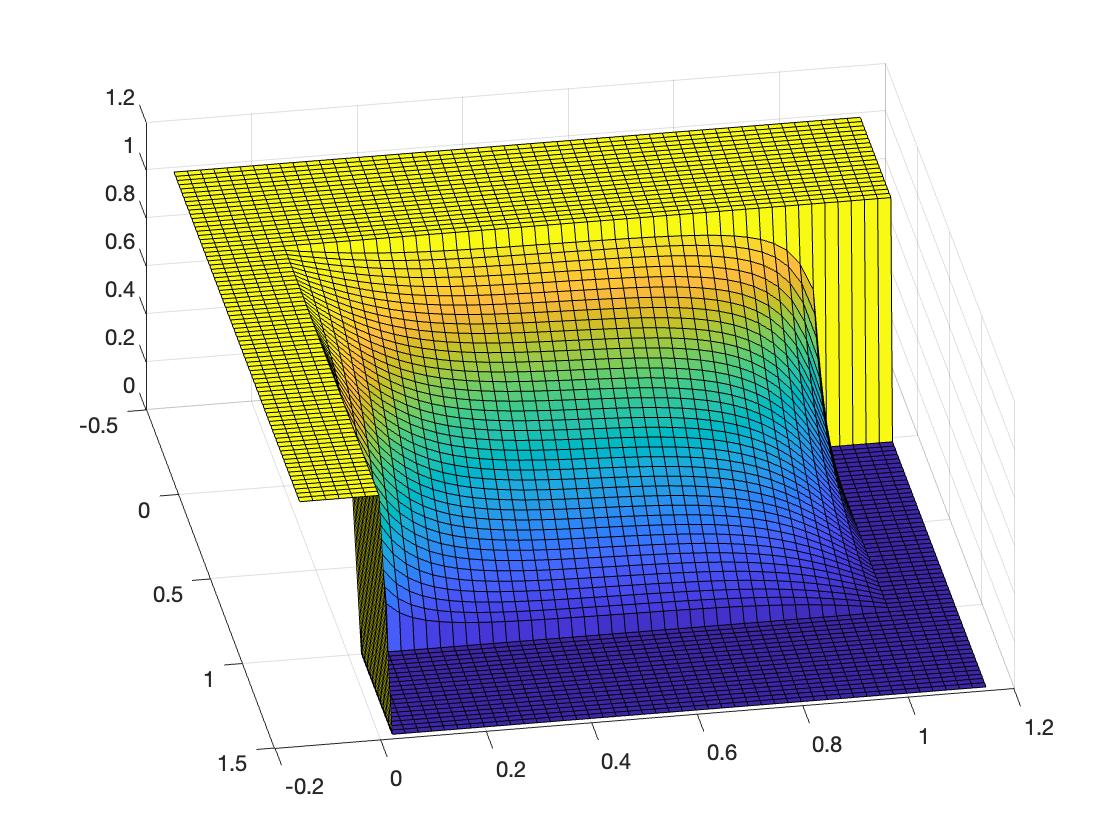}
	\end{minipage}
}
\hspace{-0mm}
\subfigure[\label{fig:bb1} The coefficient matrix ${\bf A}_2$]{
	\begin{minipage}[t]{.46\linewidth}
		\centering
		\includegraphics[width=.95\linewidth]{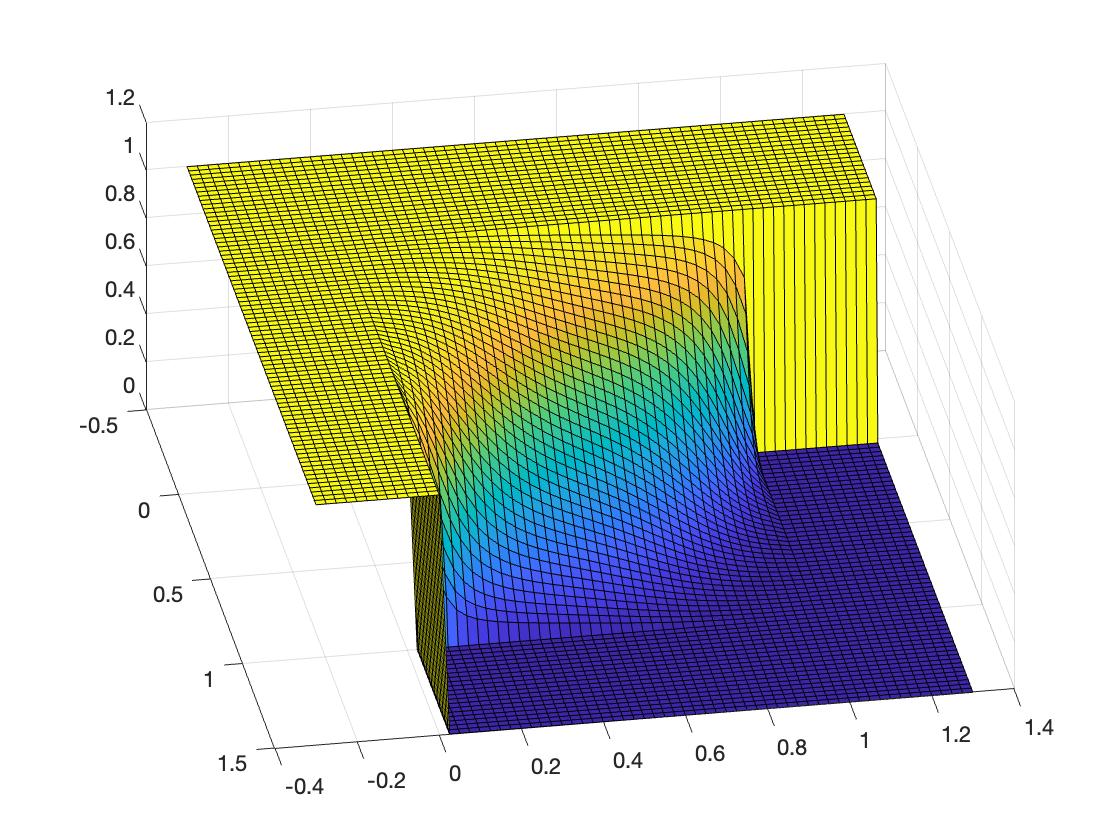}
	\end{minipage}
}
\hspace{.02mm}
\qquad
\subfigure[\label{fig:cc1} The coefficient matrix${\bf A}_3(x_1,x_2)$]{
	\begin{minipage}[t]{.46\linewidth}
		\centering
		\includegraphics[width=.95\linewidth]{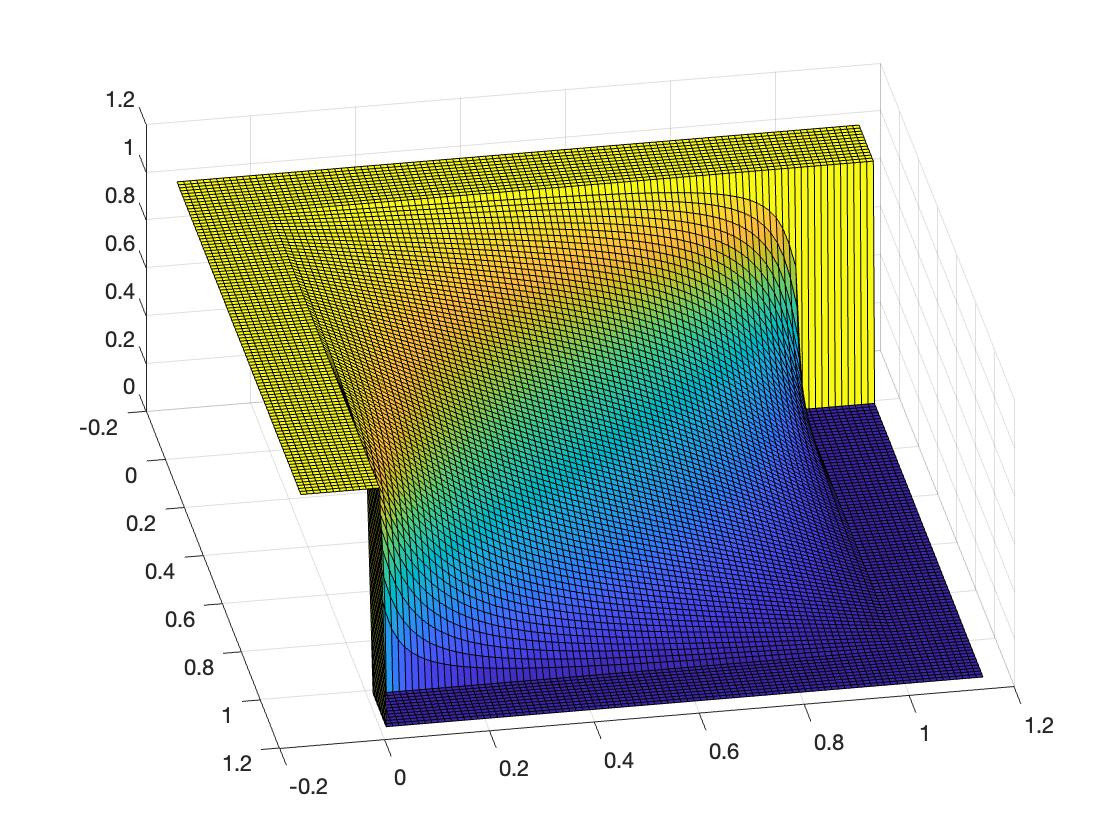}
	\end{minipage}
}
\hspace{.02mm}
\subfigure[\label{fig:dd1} The coefficient matrix ${\bf A}_4(x_1,x_2)$]{
	\begin{minipage}[t]{.46\linewidth}
		\centering
		\includegraphics[width=.95\linewidth]{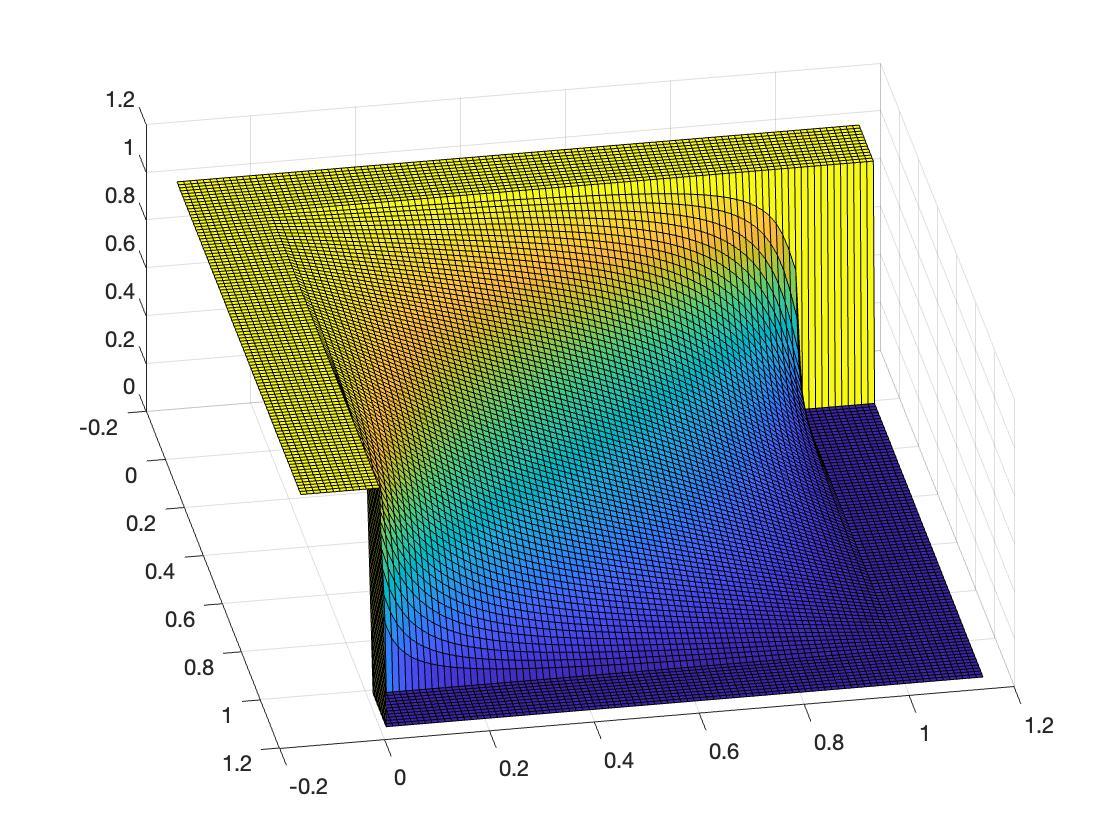}
	\end{minipage}
}}
\caption{Numerical solutions produced by the linear collocation scheme \eqref{detailed_Interpolate_linear} for the nonlocal diffusion model \eqref{dbc} with the four different diffusion coefficient matrices in Example \ref{expDis}.}
\label{talk}
\end{figure}

\subsection{Effect of $\chi_\alpha^2(d)$ on the model approximation of $\mathcal{L}_{\delta,\alpha}$ to $\mathcal{L}_{\delta}$}\label{chicomp}

In this subsection, we test the effect of the choice of $\chi_\alpha^2(d)$ on the accuracy of  the truncated nonlocal diffusion operator $\mathcal{L}_{\delta,\alpha} $ defined in  \eqref{expnon-mod} as an approximation of the original nonlocal diffusion operator $\mathcal{L}_{\delta} $ defined in \eqref{expnon}. 
\begin{exmp} \label{exmp1}
We take exactly the same experimental 
settings as those of  Example \ref{exp4}, except that we now test different values of   $\chi_\alpha^2(2)$. Specifically, we choose  $\chi_\alpha^2(d) = 9, 16, 25, 26, 49$, respectively.
\end{exmp}

We take a uniform grid of $N\times N$ cells with  the grid size $h=1/N$, where $N=40, 80, 160$ respectively. 
Table \ref{sunday} reports the numerical results on the  discrete {$L^{\infty}$}   solution errors  produced by the linear  collocation scheme \eqref{detailed_Interpolate_linear} for the nonlocal diffusion model \eqref{dbc} with  $\delta=h$ under different choices of $\chi^2_\alpha(2)$.  We can see from Table \ref{sunday} that the solution errors decrease and converge rapidly along with the increasing of $\chi^2_\alpha(2)$ from 9 to 49, and the differences of solution errors between $\chi^2_\alpha(2)=36$ and $\chi^2_\alpha(2)=49$ are almost negligible.
Therefore, we suggest the selection of {$\chi^2_\alpha(2)=36$} in practice, in order to ensure the  model accuracy of the truncated nonlocal diffusion model \eqref{dbc} while still maintaining the efficiency of numerical simulations.

\begin{table}[!htbp] \small
	\center
	\renewcommand{\arraystretch}{1.2}
	\begin{tabular}{|c||c|c|c|c|c|}
		\hline & \multicolumn{5}{|c|}{${\mathbf{A}_1}=[1,0,0;1]$} \\\hline 
		\hline  $\mathrm{N}$  &$\chi_\alpha^2(2)=9$ & $\chi_\alpha^2(2)=16$ & $\chi_\alpha^2(2)=25$ & $\chi_\alpha^2(2)=36$ & $\chi_\alpha^2(2)=49$ \\
		\hline 40 & $6.3785\times10^{-2}$ & $9.7359\times 10^{-4} $ & $4.7136\times 10^{-4}$ & $4.8765\times 10^{-4}$ & $4.6316\times10^{-4}$ \\
		\hline 80 & $1.0611\times10^{-2}$ &$6.8071\times 10^{-4}$ & $1.2696\times 10^{-4}$ & $1.1503\times 10^{-4}$ & $1.1788\times10^{-4}$ \\
		\hline 160 & $1.0517\times10^{-2}$ & $8.9142\times 10^{-4}$ & $4.2278\times 10^{-5}$ & $2.8346\times 10^{-5}$ & $2.8360 \times 10^{-5}$ \\
		\hline & \multicolumn{5}{|c|}{$\mathbf{A}_2=[10,0,0;1]$} \\\hline 
		\hline  $\mathrm{N}$  &$\chi_\alpha^2(2)=9$ & $\chi_\alpha^2(2)=16$ & $\chi_\alpha^2(2)=25$ & $\chi_\alpha^2(2)=36$ & $\chi_\alpha^2(2)=49$ \\
	\hline 40 & $1.7101\times 10^{-2}$ & $4.0173\times 10^{-3} $ & $3.7986\times 10^{-3}$ & $3.5896\times 10^{-3}$ & $3.5861\times 10^{-3}$  \\
	\hline 80 & $1.3309\times 10^{-2}$ & $1.3378\times 10^{-3}$ & $8.4862\times 10^{-4}$ & $8.4238\times 10^{-4}$ & $8.4376\times10^{-4}$ \\
	\hline 160 & $1.2548\times10^{-2}$ & $7.5558\times 10^{-4}$ & $2.0819\times 10^{-5}$ & $2.0492\times 10^{-5}$ & $2.0315 \times 10^{-5}$ \\
		\hline & \multicolumn{5}{|c|}{ $\mathbf{A}_3=[31/4,-9\sqrt{3}/4;-9\sqrt{3}/4,13/4]$} \\\hline 
		\hline  $\mathrm{N}$  &$\chi_\alpha^2(2)=9$ & $\chi_\alpha^2(2)=16$ & $\chi_\alpha^2(2)=25$ & $\chi_\alpha^2(2)=36$ & $\chi_\alpha^2(2)=49$ \\
		\hline 40 & $1.6155\times 10^{-2}$ & $3.3413\times 10^{-3} $ & $2.8710\times 10^{-3}$ & $2.8574\times 10^{-3}$ & $2.8573\times 10^{-2}$  \\
	\hline 80 & $1.3834\times 10^{-2}$ & $1.1772\times 10^{-3}$ & $6.8493\times 10^{-4}$ & $6.7127\times 10^{-4}$ & $6.7122\times10^{-4}$ \\
	\hline 160 & $1.3154\times10^{-2}$ & $6.8957\times 10^{-4}$ & $1.7647\times 10^{-4}$ & $1.6276\times 10^{-4}$ & $1.6276 \times 10^{-4}$ \\
		\hline
	\end{tabular}
	\caption{Numerical results on the  discrete  {$L^{\infty}$}   solution errors with fixed $\delta=h$ produced by the linear  collocation scheme \eqref{detailed_Interpolate_linear} for the nonlocal diffusion model \eqref{dbc} under different choices of $\chi^2_\alpha(d)$  in Example \ref{exmp1}.}
	\label{sunday}
\end{table}
\section{Conclusions}

This paper presents a novel bond-based nonlocal diffusion model with matrix-valued coefficients in non-divergence form. Our approach involves integrating the coefficient matrix into the covariance matrix and employing the multivariate Gaussian function with truncation as the kernel function to accurately encapsulate  the diffusion process. 
Substantiating the robustness of our model, we establish its well-posedness along with elucidating certain inherent properties.  To numerically solve the model, we also design an efficient linear collocation discretization scheme. The paper includes extensive experiments conducted in two and three dimensions to showcase the versatility of our model in addressing various isotropic and anisotropic diffusion problems. Furthermore, these experiments  numerically demonstrate high-order accuracy and effective asymptotic compatibility of our collocation scheme
for solving the proposed nonlocal diffusion model. On the other hand,  it still remains an open question to rigorously prove the related results on convergence and asymptotic compatibility of the proposed collocation scheme in general dimensions. In addition,  it is  highly expected to further develop bond-based nonlocal diffusion models in conservative form.
\label{}

\small

\end{document}